\documentclass[11pt]{article}

\usepackage[english]{babel}
\usepackage[a4paper,top=3cm,bottom=3cm,left=2.5cm,right=2.5cm]{geometry}
\usepackage[utf8]{inputenc}
\usepackage[T1]{fontenc}
\usepackage[all,cmtip]{xy}
\usepackage{hyperref}

\usepackage{amsmath,amssymb,amsthm}	
\usepackage[mathscr]{eucal}			
\usepackage{amsmath,amssymb}
\usepackage{amsthm}
\usepackage{amsfonts}
\usepackage{enumerate}
\usepackage{mathabx}
\usepackage[all,cmtip]{xy}
\usepackage{amscd}

\usepackage[svgnames]{xcolor}
\usepackage{listings}

\usepackage{graphicx}
\usepackage{verbatim}

\usepackage{tikz}
\usepackage{tikz-cd}
\usetikzlibrary{arrows,calc,matrix,patterns,shadings,backgrounds,fit}

\makeatletter
\newcommand*\bigcdot{\mathpalette\bigcdot@{.8}}
\newcommand*\bigcdot@[2]{\mathbin{\vcenter{\hbox{\scalebox{#2}{$\m@th#1\bullet$}}}}}

\theoremstyle{plain} 
\newtheorem{thm}{Theorem}[section]
\newtheorem{cor}[thm]{Corollary} 
\newtheorem{lemma}[thm]{Lemma} 
\newtheorem{prop}[thm]{Proposition}

\theoremstyle{definition} 
\newtheorem{dfn}[thm]{Definition}
\newtheorem{rmk}[thm]{Remark} 

\newtheorem*{rw}{Related works}
\newtheorem*{mot}{Further questions}
\newtheorem*{plan}{Plan}
\newtheorem*{ack}{Acknowledgements}

\DeclareMathOperator{\C}{\mathbb{C}}

\DeclareMathOperator{\Q}{\mathbb{Q}}		
\DeclareMathOperator{\Z}{\mathbb{Z}}
\DeclareMathOperator{\R}{\mathbb{R}}

\DeclareMathOperator{\Hom}{\text{Hom}}

\DeclareMathOperator{\A}{\mathcal{A}}
\DeclareMathOperator{\T}{\mathcal{T}}
\DeclareMathOperator{\D}{\text{D}^b}

\DeclareMathOperator{\DD}{\emph{D}^{\emph{b}}}
\DeclareMathOperator{\ch}{\text{ch}}
\DeclareMathOperator{\td}{\text{td}}

\DeclareMathOperator{\Ku}{\mathsf{Ku}}

\DeclareMathOperator{\Coh}{\text{Coh}}
\DeclareMathOperator{\id}{\text{id}}

\def\P{\ensuremath{\mathbb{P}}}
\def\L{\ensuremath{\mathbb L}}
\def\Hom{\mathop{\mathrm{Hom}}\nolimits}

\def\Stab{\mathop{\mathrm{Stab}}\nolimits}
\DeclareMathOperator{\CCoh}{\emph{Coh}}

\def\AA{\ensuremath{\mathcal A}}

\def\CC{\ensuremath{\mathcal C}}
\def\DD{\ensuremath{\mathcal D}}

\def\FF{\ensuremath{\mathcal F}}

\def\HH{\ensuremath{\mathcal H}}
\def\II{\ensuremath{\mathcal I}}
\def\JJ{\ensuremath{\mathcal J}}
\def\KK{\ensuremath{\mathcal K}}

\def\NN{\ensuremath{\mathcal N}}
\def\OO{\ensuremath{\mathcal O}}
\def\PP{\ensuremath{\mathcal P}}

\def\TT{\ensuremath{\mathcal T}}

\def\ch{\mathop{\mathrm{ch}}\nolimits}

\def\rk{\mathop{\mathrm{rk}}}

\def\lHom{\mathop{\mathcal Hom}\nolimits}

\def\Aut{\text{Aut}}

\newcommand\blfootnote[1]{%
  \begingroup
  \renewcommand\thefootnote{}\footnote{#1}%
  \addtocounter{footnote}{-1}%
  \endgroup
}

\title{Some remarks on Fano threefolds of index two and stability conditions}
\author{Laura Pertusi and Song Yang}
\date{\today}

\begin{document}
\maketitle
\begin{abstract}
We prove that ideal sheaves of lines in a Fano threefold $X$ of Picard rank one and index two are stable objects in the Kuznetsov component $\Ku(X)$, with respect to the stability conditions constructed by Bayer, Lahoz, Macrì and Stellari, giving a modular description to the Hilbert scheme of lines in $X$. When $X$ is a cubic threefold, we show that the Serre functor of $\Ku(X)$ preserves these stability conditions. As an application, we obtain the smoothness of non-empty moduli spaces of stable objects in $\Ku(X)$. When $X$ is a quartic double solid, we describe a connected component of the stability manifold parametrizing stability conditions on $\Ku(X)$.
\end{abstract}

\blfootnote{\textup{2010} \textit{Mathematics Subject Classification}: \textup{14J35}, \textup{14J45}, \textup{18E30}.}
\blfootnote{\textit{Key words and phrases}: Fano threefolds, Bridgeland stability conditions, semiorthogonal decompositions, moduli spaces, cubic threefolds.}

\section{Introduction}

The notion of stability condition on a triangulated category has been introduced by Bridgeland in \cite{Bri}. One of the main powerful aspect in this theory is that the set parametrizing stability conditions on a triangulated category has a natural topology, which endows it of the structure of a complex manifold. An interesting question is to understand the properties of the stability manifold, e.g.\ if it is non-empty, simply-connected or connected, or giving a description of a connected component. 

Even in the geometric setting, considering the case of the bounded derived category $\D(X)$ of coherent sheaves on a smooth projective variety $X$, these questions are very hard to treat. A complete description of the stability manifold is known only if $X$ is a smooth projective curve by \cite{Bri}, \cite{Ok} and \cite{Macri}, where the authors consider elliptic curves, $\P^1$, and curves of genus $\geq 1$, respectively. A connected component of the stability manifold of a K3 or abelian surface is described in \cite{Bri2} (see also \cite{HMS} for the case of twisted K3 or abelian surfaces, and \cite{BaBri} for a further description for K3 surfaces of Picard rank one). More generally, in \cite{AB} the authors construct a family of stability conditions when $X$ is a surface. In dimension three, stability conditions are constructed for Fano threefolds (see for example \cite{BMT}, \cite{Macri2} for the projective space, \cite{LiFano} when the Picard rank is one, \cite{BMSZ}, \cite{P} for the general case), abelian threefolds (see \cite{BMS}, \cite{MP} and \cite{MP2}), some resolutions of finite quotients of abelian threefolds (see \cite{BMS}) and quintic threefolds (see \cite{Li}). See also \cite{Liu} for the construction of stability conditions on products with curves.

On the other hand, having stability conditions, it is possible to consider moduli spaces of stable complexes and investigate their properties, like non-emptyness, projectivity, smoothness. We recall that if $X$ is a K3 or abelian surface, then Inaba generalized in \cite{Inaba} Mukai's smoothness result in \cite{Mukai} to moduli spaces of simple objects in $\D(X)$. 

Recently, in \cite{BLMS} Bayer, Lahoz, Macr{\`{\i}} and Stellari have introduced a general criterion to induce stability conditions on the right orthogonal of an exceptional collection in a triangulated category $\mathscr{T}$, from a weak stability condition on $\mathscr{T}$ (see Section 2.2). They applied this result to the case of Fano threefolds of Picard rank one and of cubic fourfolds. As another application, we mention the construction of stability conditions on Gushel-Mukai varieties in \cite{PPZ}.

In this paper, we focus on the case of a Fano threefold $X$ of Picard rank one and index two. The bounded derived category has a semiorthogonal decomposition of the form
$$
\D(X)=\langle \Ku(X), \mathcal{O}_{X}, \mathcal{O}_{X}(H)\rangle,
$$
where $H:=-\frac{1}{2}K_X$. Here, by definition $\Ku(X)$ is the right orthogonal complement of the line bundles $\mathcal{O}_{X}$ and $\mathcal{O}_{X}(H)$, i.e.\
$$\Ku(X):= \lbrace E \in \D(X): \Hom_{\D(X)}(\mathcal{O}_X,E[p])=\Hom_{\D(X)}(\mathcal{O}_X(H),E[p])=0, \forall p\in \mathbb{Z} \rbrace.$$
The subcategory $\Ku(X)$ is called the \emph{Kuznetsov component}. We denote by $\sigma(\alpha,\beta)$ the stability conditions on $\Ku(X)$ constructed in \cite{BLMS}. As recalled in Theorem \ref{thm_U}, the values of $\alpha$ and $\beta$ vary in the set
$$V:= \lbrace (\alpha,\beta) \in \R_{>0} \times \R : -\frac{1}{2} \leq \beta < 0, \alpha < -\beta, \text{ or } -1 < \beta <-\frac{1}{2}, \alpha \leq 1+\beta \rbrace.$$
%0< \alpha < \text{min}\lbrace-\beta, \beta+1 \rbrace, -1 <\beta <0 \rbrace
In Proposition \ref{prop_conncomp}, we show that the stability conditions $\sigma(\alpha,\beta)$ parametrized by $V$ are in the same orbit $\KK$ with respect to the right action of the universal covering space $\tilde{\mathrm{GL}}^+_2(\R)$ of $\mathrm{GL}_2^+(\R)$ on the stability manifold $\Stab(\Ku(X))$ of $\Ku(X)$.

The first result gives an interpretation of the Hilbert scheme of lines in $X$ as a moduli space of objects in $\Ku(X)$ which are stable with respect to a stability condition in the orbit $\KK$.

\begin{thm}\label{Fanolines_modspace}
Let $X$ be a Fano threefold of Picard rank one and index two. If $X$ has degree $\neq 1$, then for any $\sigma\in \KK$,
the Hilbert scheme of lines $\Sigma(X)$ in $X$ is isomorphic to a moduli space $M_\sigma(\Ku(X),[\II_\ell])$
of $\sigma$-stable objects in $\Ku(X)$ 
with the same numerical class as the ideal sheaf of a line in $X$. If $X$ has degree $1$, then $\Sigma(X)$ is an irreducible component of $M_\sigma(\Ku(X),[\II_\ell])$.
\end{thm}

In the second part, we consider $X$ of degree $3$, i.e.\ a cubic threefold. We show an analogous of Mukai's smoothness result in this setting.  

\begin{thm}
\label{cor_smoothmod}
If $X$ is a cubic threefold, then non-empty moduli spaces of stable objects in $\Ku(X)$ with respect to a stability condition in $\KK$ are smooth.
\end{thm}

\noindent A key point in the proof of Theorem \ref{cor_smoothmod} is the fact that the Serre functor of $\Ku(X)$ preserves the orbit $\KK$, as shown in Corollary \ref{cor_serrefunctorc3}. As another application of this property, in Theorem \ref{thm_catTorelli} we give an alternative proof of the categorical Torelli Theorem proved in \cite[Theorem 1.1]{BMMS}.

Finally, in the degree-$2$ case, we describe a connected component of the stability manifold of $\Ku(X)$.

\begin{thm}
\label{thm_conncomp}
Let $X$ be a quartic double solid. Then the orbit $\KK$ is a connected component of maximal dimension of $\Stab(\Ku(X))$.
\end{thm}

\begin{rw}
The first example of a stability condition on the Kuznetsov component of a cubic threefold is given in \cite{BMMS}. In particular, the authors prove that ideal sheaves of lines are stable with respect to this stability condition and that they are the only stable objects with numerical class $[\II_\ell]$. In Propositions \ref{lemma_idealsheafstabinK} and \ref{lem_numclassI} we prove that the same results hold for the stability conditions in the connected component $\KK$, containing those constructed in \cite{BLMS} and for every degree $d \neq 1$. If $d=1$, we show the stability of ideal sheaves of lines in Proposition \ref{lemma_idealsheafstabinK}. 
\indent The $d=1$ case is further investigated in the recent preprint \cite{PeRo}, where the authors classify all the stable objects in the moduli space containing the Hilbert scheme of lines. 

In \cite[Theorem 1.1]{LPZ} the analogous of Theorem \ref{Fanolines_modspace} is proved in the case of the Fano variety of lines in a cubic fourfold.

The analogous of Inaba's smoothness result has been proved in \cite{LiZhao2} for smooth projective surfaces with a Poisson structure and in \cite{BLM+} for moduli spaces of simple complexes in the Kuznetsov component of a cubic fourfold. In the upcoming paper \cite{Perry} this result is generalized to families of two-dimensional Calabi-Yau categories.

In \cite{APR} the authors study certain moduli spaces of stable objects in $\Ku(X)$ with torsion class in the numerical Grothendieck group, with application to Torelli type questions in the case of quartic double solids.
\end{rw}

\begin{mot}
The category $\Ku(X)$ can be considered as a non-commutative curve, e.g.\ the numerical Grothendieck group of $\Ku(X)$ is the same of that of a curve. We hope that the results in this paper could be useful to understand whether $\Stab(\Ku(X))$ has a unique connected component when the degree of $X$ is $2$ or $3$, completing the analogy with curves. Note that this is already known in the degree-$4$ case applying \cite[Theorem 2.7]{Macri}, as the Kuznetsov component is equivalent to the bounded derived category of a genus-$2$ curve by \cite[Theorem 4.4]{Kuz}. Moreover, if $X$ has degree $5$, then $\Stab(\Ku(X))$ is completely described by \cite[Theorem 1.1]{DK19} (see Remark \ref{rmk_d=4or5}).

In the case of cubic threefolds, it would be interesting to understand if the stability condition $\bar{\sigma}$ constructed in \cite{BMMS} is in the orbit $\KK$ containing the stability conditions constructed in \cite{BLMS}. Theorem \ref{Fanolines_modspace} and, more generally, the properties proved in Section 5 give an evidence to this guess.
\end{mot}

\begin{plan}
In Section 2 we review the definition of (weak) stability conditions on a triangulated category and their construction in the case of $\D(X)$. In Section 3 we discuss the method to induce stability conditions on the Kuznetsov component of a Fano threefold $X$ of index two and Picard rank one introduced in \cite{BLMS} and we prove that these stability conditions are in the same orbit $\KK$ with respect to the $\tilde{\mathrm{GL}}^+_2(\R)$-action (Proposition \ref{prop_conncomp}). In Section 4 we prove Theorem \ref{Fanolines_modspace}, showing that ideal sheaves of lines are $\sigma$-stable and that if $X$ has degree $\neq 1$, the only $\sigma$-stable objects with that numerical class are ideal sheaves of lines for $\sigma \in \KK$ (Propositions \ref{lemma_idealsheafstabinK} and \ref{lem_numclassI}). Section 5 is devoted to cubic threefolds. We prove that the Serre functor of the Kuznetsov component preserves the orbit $\KK$ (Corollary \ref{cor_serrefunctorc3}). This computation is rather technical and based on \cite[Lemma 3]{LiZhao2} (see also \cite[Lemma 3.1]{FLLQ}), which allows to control the phase of semistable factors when deforming the stability condition. After explaining some consequences of this result, we prove Theorem \ref{cor_smoothmod} and the categorical Torelli Theorem \ref{thm_catTorelli}. In Section 6 we consider quartic double solids and we prove Theorem \ref{thm_conncomp}.
\end{plan}

\begin{ack}
It is a pleasure to thank Arend Bayer, Chunyi Li, Emanuele Macrì, Paolo Stellari and Xiaolei Zhao for many useful discussions and comments on the preliminary version of this work. We would especially like to thank Arend Bayer for suggesting the application in Theorem \ref{cor_smoothmod} and Chunyi Li for his help in the proof of Lemma \ref{lemma_secondtilt}. We are also grateful to Marin Petkovi\'c and Franco Rota for pointing out a mistake in Theorem \ref{Fanolines_modspace} in the previous version of the paper. Finally, we thank the referees for careful reading of the paper and useful suggestions.

This paper was written when the second author was visiting the Department of Mathematics of the Università degli Studi di Milano funded by the China Scholarship Council and the first author was visiting the Max-Planck-Institut f\"ur Mathematik in Bonn, the University of Edinburgh and the University of Warwick. We thank these institutions for their hospitality.

The first author is supported by the ERC Consolidator Grant ERC-2017-CoG-771507, StabCondEn.
The second author is partially supported by the NSFC Grant 11701414.
\end{ack}

\section{Construction of (weak) stability conditions}

In this section we review the definition of (weak) stability conditions on a triangulated category, the general criterion to induce stability conditions on Kuznetsov components proved in \cite[Proposition 5.1]{BLMS} and the construction of weak stability conditions by double tilting on $\D(X)$ given in \cite[Section 2]{BLMS}.

\subsection{(Weak) stability conditions}

A (weak) stability condition on a triangulated category $\mathscr{T}$ is essentially given by a heart of a bounded t-structure and a (weak) stability function, satisfying some conditions. Let us recall precisely these notions.

\begin{dfn}
A {\it heart of a bounded $t$-structure} on $\mathscr{T}$ is a full subcategory $\A\subset \mathscr{T}$ such that
\begin{enumerate}
\item[(i)]  for $E, F\in \A$ and $k<0$ we have $\Hom(E, F[k])=0$, and
\item[(ii)]  for every object $E\in \mathscr{T}$ there is a sequence of morphisms
$$
\xymatrix@C=0.5cm{
  0 =E_0 \ar[r]^{\phi_{1}}& E_{1} \ar[r]^{} & \cdots  \ar[r]^{\phi_{m}\;\;\;\;\;} & E_{m} =E }
$$
such that $\mathrm{Cone}(\phi_{i})$ is of the form $A_{i}[k_i]$ for some sequence $k_1>k_2>\cdots>k_m$
of integers and objects $A_i\in \A$.
\end{enumerate}
\end{dfn}

\begin{dfn}
\label{def_weakstab}
Let $\A$ be an abelian category. A {\it weak stability function} on $\A$ is a group homomorphism 
$$
\begin{array}{cccl}
Z:& K(\A)&\longrightarrow& \C \\
&E& \longmapsto& \Re Z(E) +i \Im Z(E),
\end{array}
$$
where $K(\A)$ denotes the Grothendieck group of $\A$, such that for every non-zero object $E\in \A$, we have 
$$
\Im Z(E)\geq 0, \; \textrm{and}\; \Im Z(E)= 0 \Rightarrow \Re Z(E)\leq 0.
$$
We say that $Z$ is a {\it stability function} on $\A$ if in addition for $E \neq 0$, $\Im Z(E)= 0$ implies $\Re Z(E)< 0$.
\end{dfn}

Fix a finite rank lattice $\Lambda$ and 
a surjective homomorphism $\upsilon: K(\mathscr{T})\twoheadrightarrow\Lambda$.

\begin{dfn}
A {\it weak stability condition on $\mathscr{T}$} with respect to the lattice $\Lambda$ is a pair $\sigma=(\A, Z)$, where $\A$ is the heart of a bounded $t$-structure and $Z: \Lambda \rightarrow \mathbb{C}$ is a group homomorphism, such that the following conditions hold:
\begin{enumerate}
\item[(i)] The composition $K(\A)=K(\mathscr{T})\xrightarrow{\upsilon} \Lambda\xrightarrow{Z} \mathbb{C}$ is a weak stability function on $\A$; for simplicity, we denote $Z(E):=Z(\upsilon(E))$.
The function $Z$ allows to define a {\it slope} for any object $E\in \A$ by
$$
\mu_{\sigma}(E)
:=
\begin{cases}
-\frac{\Re Z(E)}{\Im Z(E)} & \textrm{if}\, \Im Z(E)>0; \\
+\infty &  \textrm{otherwise};
\end{cases}
$$
and a notion of {\it stability} : an object $0\neq E\in \A$ is called {\it $\sigma$-semistable} (resp. {\it $\sigma$-stable}) if for every non-zero proper subobject $F\subset E$, we have
$\mu_{\sigma}(F)\leq\mu_{\sigma}(E)$ (resp. $\mu_{\sigma}(F)<\mu_{\sigma}(E/F)$). 

\item[(ii)] (HN-filtrations) Any object of $\A$ has a Harder-Narasimhan filtration in $\sigma$-semistable ones.

\item[(iii)] (Support property) There is a quadratic form $Q$ on $\Lambda\otimes \mathbb{R}$ such that $Q|_{\ker Z}$ is negative definite, and $Q(E)\geq 0$ for all $\sigma$-semistable objects $E\in \A$.
\end{enumerate}
\end{dfn}

\begin{dfn}
A weak stability condition $\sigma=(\A, Z)$ on $\mathscr{T}$ with respect to the lattice $\Lambda$ is called a {\it Bridgeland stability condition} if $Z$ is a stability function.
\end{dfn}

We need to introduce some terminology we will use in the following. Let $\sigma$ be a (weak) stability condition for $\mathscr{T}$. 
\begin{dfn}
\label{def_slicing}
The \textit{phase} of a $\sigma$-semistable object $E \in \A$ is
$$\phi(E):=\frac{1}{\pi}\text{arg}(Z(E)) \in (0,1].$$
If $Z(E)=0$, then $\phi(E)=1$. For $F=E[n]$, we set
$$\phi(E[n]):=\phi(E)+n.$$ 
A \textit{slicing} $\PP$ of $\mathscr{T}$ is a collection of full additive subcategories $\PP(\phi) \subset \mathscr{T}$ for $\phi \in \R$, such that:
\begin{enumerate}
\item[(i)] for $\phi \in (0,1]$, the subcategory $\PP(\phi)$ is given by the zero object and all $\sigma$-semistable objects with phase $\phi$;
\item[(ii)] for $\phi+n$ with $\phi \in (0,1]$ and $n \in \Z$, we set $\PP(\phi+n):=\PP(\phi)[n]$.
\end{enumerate}
\end{dfn}

\noindent We will both use the notation $\sigma=(\AA,Z)$ and $\sigma=(\PP,Z)$ for a (weak) stability condition with heart $\A=\PP((0,1])$, where $\PP$ is a slicing. 

We denote by $\Stab(\mathscr{T})$ the set of stability conditions on $\mathscr{T}$. A very deep result of Bridgeland is that $\Stab(\mathscr{T})$ is actually a complex manifold, as stated below.  
\begin{thm}[Bridgeland Deformation Theorem, \cite{Bri}]
The continuous map $\mathcal{Z}: \Stab(\mathscr{T}) \to \Hom(\Lambda,\C)$ defined by $(\A,Z) \mapsto Z$, is a local homeomorphism. In particular, the topological space $\Stab(\mathscr{T})$ has the structure of a complex manifold of dimension $\rk(\Lambda)$.
\end{thm}

Recall that the universal covering space $\tilde{\mathrm{GL}}^+_2(\R)$ of $\mathrm{GL}^+_2(\R)$ has a right action on $\Stab(\mathscr{T})$, defined as follows. For $\tilde{g}=(g,M) \in \tilde{\mathrm{GL}}^+_2(\R)$, where $g: \R \to \R$ is an increasing function such that $g(\phi+1)=g(\phi)+1$ and $M \in \mathrm{GL}^+_2(\R)$, and $\sigma=(\PP,Z) \in \Stab(\mathscr{T})$, we have that $\sigma \cdot \tilde{g}=(\PP',Z')$ is a stability condition with $Z'=M^{-1} \circ Z$ and $\PP'(\phi)=\PP(g(\phi))$ (see \cite[Lemma 8.2]{Bri}). We will sometimes use the notation $\PP'=\PP \cdot \tilde{g}$. Also the group of linear exact autoequivalences $\Aut(\mathscr{T})$ of $\mathscr{T}$ acts on the left of $\Stab(\mathscr{T})$ by $\Phi \cdot \sigma=(\Phi(\PP), Z \circ \Phi_*^{-1})$, where $\Phi \in \Aut(\mathscr{T})$ and $\Phi_*$ is the automorphism of $K(\mathscr{T})$ induced by $\Phi$.

The construction of Bridgeland stability conditions is in general a difficult task. However, starting from a weak stability condition $\sigma=(\A, Z)$ on $\mathscr{T}$, it is possible to produce a new heart of a bounded t-structure, by \textit{tilting} $\A$. Let us recall this method. Let $\mu \in \R$; we define the following subcategories of $\A$:
\begin{align*}
\TT^{\mu}_{\sigma}&:= \lbrace E \in \A: \text{all HN factors } F \text{ of }E \text{ have slope } \mu_{\sigma}(F)>\mu \rbrace\\
                  & = \langle E \in \A: E \text{ is } \sigma\text{-semistable with }\mu_{\sigma}(E) >\mu \rangle
\end{align*}
and 
\begin{align*}
\FF^{\mu}_{\sigma} &:= \lbrace E \in \A: \text{all HN factors } F \text{ of }E \text{ have slope } \mu_{\sigma}(F)\leq\mu \rbrace\\
                   &= \langle E \in \A: E \text{ is } \sigma\text{-semistable with }\mu_{\sigma}(E) \leq \mu \rangle.
\end{align*}
Here, the symbol $\langle - \rangle$ means the extension closure, i.e.\ the smallest full additive subcategory of $\T$ containing the objects in the brackets which is closed with respect to extensions. 

\begin{prop}[\cite{HRS}]
\label{prop_HRS}
The category
$$\A_{\sigma}^{\mu}:= \langle \TT^{\mu}_{\sigma},\FF^{\mu}_{\sigma}[1] \rangle$$
is the heart of a bounded t-structure on $\TT$.
\end{prop}

\noindent We say that the heart $\A_{\sigma}^{\mu}$ is obtained by tilting $\A$ with respect to the weak stability condition $\sigma$ at the slope $\mu$. In Section \ref{section_wsX}, we will explain how to construct weak stability conditions on $\D(X)$ by tilting $\Coh(X)$ with respect to slope stability.

\subsection{Inducing stability conditions}

Let $\mathscr{T}$ be a triangulated category with Serre functor $S_{\mathscr{T}}$.
If $\{E_0,E_1,\cdots ,E_l\}$ is an exceptional collection in $\mathscr{T}$, then there exists a semiorthogonal decomposition of the form
$$
\mathscr{T}
=\langle \mathscr{D}_1, \mathscr{D}_2\rangle,
$$
where $\mathscr{D}_2:=\langle E_0,E_1,\cdots ,E_l\rangle$ and $\mathscr{D}_{1}:=\mathscr{D}_2^{\bot}$. The next proposition gives a criterion in order to induce a stability condition on $\mathscr{D}_1$ from a weak stability condition on $\mathscr{T}$.

\begin{prop}[{\cite[Proposition 5.1]{BLMS}}]
\label{prop_inducstab}
Let $\sigma=(\A, Z)$ be a weak stability condition on $\mathscr{T}$. 
Assume that the exceptional collection $\{E_0,E_1,\cdots ,E_l\}$ satisfies the following conditions: 
\begin{enumerate}
\item $E_i\in \A$;
\item $S_{\mathscr{T}}(E_i)\in \A[1]$; 
\item  $Z(E_i)\neq0$ for all $i=0,1, \cdots, l$.
\end{enumerate}
If moreover there are no objects $0\neq F\in \A_1:=\A\cap \mathscr{D}_1$ with $Z(F)=0$, i.e., $Z_1:=Z|_{\A_1}$ is a stability function on $\A_1$,
then the pair $\sigma_{1}=(\A_1,Z_1)$ is a stability condition on $\mathscr{D}_1$.
\end{prop}

\subsection{Weak stability conditions on $\D(X)$}
\label{section_wsX}
Let $X$ be a smooth projective variety of dimension $n$ 
and $H$ be an ample divisor on $X$. Following \cite[Section 2]{BLMS}, we review the construction of weak stability conditions on $\D(X)$.

For any $j\in \{0,1, 2,\cdots ,n\}$, 
consider the lattice $\Lambda_{H}^{j}\cong \Z^{j+1}$ generated by 
$$
(H^{n}\ch_{0}, H^{n-1}\ch_{1}, \cdots, H^{n-j}\ch_{j})\in \Q^{j+1},
$$
with the surjective map $\upsilon_{H}^{j}: K(X)\rightarrow \Lambda_{H}^{j}$ induced by the Chern character.
Then the pair 
$$\sigma_{H}:=(\Coh(X), Z_{H}),$$
where $Z_H: \Lambda_H^1 \to \C$ is given by 
$$
Z_{H}(E):=-H^{n-1}\ch_{1}(E)+i H^{n}\ch_{0}(E),
$$
defines a weak stability condition on $\D(X)$ with respect to the lattice $\Lambda^1_H$ (see \cite[Example 2.8]{BLMS}). Note that the slope $\mu_H$ defined by the weak stability function $Z_H$ coincides with the classical notion of {\it slope stability}. Moreover, any $\mu_{H}$-semistable sheaf $E$ satifies the following Bogomolov-Gieseker inequality:
\begin{equation*}
\Delta_{H}(E):=(H^{n-1}\ch_{1}(E))^{2}-2H^{n}\ch_{0}(E)\cdot H^{n-2}\ch_{2}(E)\geq 0.
\end{equation*}

Given a parameter $\beta\in \R$,
we denote by 
$$\Coh^{\beta}(X)$$
the heart of a bounded $t$-structure obtained by tilting the weak stability condition $\sigma_H$ at the slope $\mu_H=\beta$. 
For $E \in \D(X)$, we set
$$\ch^{\beta}(E):=e^{-\beta}\ch(E).$$
Explicitly, the first three terms are
$$\ch_0^{\beta}(E):=\ch_0(E), \quad  \ch_1^{\beta}(E):=\ch_1(E)-\beta H \ch_0(E)$$
and
$$\ch_2^{\beta}(E):= \ch_2(E) -\beta H \ch_1(E) +\frac{\beta^2 H^2}{2}\ch_0(E).$$

\begin{prop}[\cite{BLMS}, Proposition 2.12]
\label{first-tilting-wsc}
For any $(\alpha, \beta)\in \R_{>0}\times \R$,
the pair 
$$\sigma_{\alpha, \beta}=(\emph{Coh}^{\beta}(X), Z_{\alpha, \beta})$$
with
$$
Z_{\alpha, \beta}(E)
:=\frac{1}{2}\alpha^2 H^{n}\ch_{0}^{\beta}(E)-H^{n-2}\ch_{2}^{\beta}(E)
+i H^{n-1}\ch_{1}^{\beta}(E)
$$
defines a weak stability condition on $\emph{D}^b(X)$ with respect to $\Lambda_{H}^{2}$. 
The quadratic form $Q$ can be given by the discriminant $\Delta_{H}$.
Moreover, these weak stability conditions vary continuously as $(\alpha, \beta)\in \R_{>0}\times \R$ varies.
\end{prop}

By definition the slope with respect to $Z_{\alpha,\beta}$ is
$$\mu_{\alpha,\beta}(E)=-\frac{\Re Z_{\alpha,\beta}(E)}{\Im Z_{\alpha,\beta}(E)} \quad \text{for }\Im Z_{\alpha,\beta}(E) \neq 0.$$
We can visualize the weak stability conditions $\sigma_{\alpha,\beta}$ in the upper half plane $$\lbrace (\alpha,\beta) \in \R \times \R: \alpha>0 \rbrace.$$
\begin{dfn}
Let $v$ be a vector in $\Lambda_H^2$. 
\begin{enumerate}
\item A \textit{numerical wall} for $v$ is the set of pairs $(\alpha,\beta) \in \R_{>0} \times \R$ such that there is a vector $w \in \Lambda_H^2$ verifying the numerical relation
$\mu_{\alpha,\beta}(v)=\mu_{\alpha,\beta}(w)$.
\item A \textit{wall} for $F \in \Coh^\beta(X)$ is a numerical wall for $v:=\ch_{\leq 2}(F)$, where $\ch_{\leq 2}(F):=(\ch_0(F),\ch_1(F),\ch_2(F))$, such that for every $(\alpha,\beta)$ on the wall there is an exact sequence of semistable objects $0 \to E \to F \to G \to 0$ in $\Coh^{\beta}(X)$ such that $\mu_{\alpha,\beta}(F)=\mu_{\alpha,\beta}(E)=\mu_{\alpha,\beta}(G)$ gives rise to the numerical wall.
\item  A \textit{chamber} is a connected component in the complement of the union of walls in the upper half plane. 
\end{enumerate}
\end{dfn}

\noindent A key property is that the weak stability conditions $\sigma_{\alpha,\beta}$ satisfy well-behaved wall-crossing: walls with respect to a class $v \in \Lambda_H^2$ are locally finite. In particular, if $v=\ch_{\leq 2}(E)$ with $E \in \Coh^{\beta}(X)$, then the stability of $E$ remains unchanged as $(\alpha,\beta)$ varies in a chamber by \cite[Proposition B.5]{BMS}.

We end this section by recalling the following variant of the weak stability conditions of Proposition \ref{first-tilting-wsc}, which will be used in the next sections. Fix $\mu\in \R$ and let $u$ be the unit vector in the upper half plane with $\mu=-\frac{\Re u}{\Im u}$. We denote by
$$
\Coh_{\alpha, \beta}^{\mu}(X)
$$
the heart obtained by tilting the weak stability condition $\sigma_{\alpha, \beta}=(\Coh^{\beta}(X), Z_{\alpha, \beta})$ at the slope $\mu_{\alpha,\beta}=\mu$.
\begin{prop}[{\cite[Proposition 2.15]{BLMS}}]
\label{prop_secondtiltstab}
The pair $\sigma_{\alpha, \beta}^{\mu}:=(\emph{Coh}_{\alpha, \beta}^{\mu}(X), Z_{\alpha, \beta}^{\mu})$, where 
$$
Z_{\alpha, \beta}^{\mu}:=\frac{1}{u} Z_{\alpha, \beta},
$$
is a weak stability condition on $\emph{D}^b(X)$.
\end{prop}

\section{Fano threefolds of Picard rank $1$ and index $2$}

In this section we explain how to induce stability conditions on the Kuznetsov component of a Fano threefold of Picard rank one and index two, as proved in \cite[Section 6]{BLMS}. Then we prove that these induced stability conditions are in the same orbit with respect to the $\tilde{\mathrm{GL}}^+_2(\R)$-action on $\Stab(\Ku(X))$. 

\subsection{Classification and Kuznetsov component}
\label{sec_classific}
Let $X$ be a Fano threefold with $\text{Pic}(X)\cong \Z$. Assume that $X$ has index $2$, i.e.\ $K_X=-2H$, where $H$ is a positive generator of $\text{Pic}(X)$. The degree of $X$ is $d:=H^3$. Recall the following classification result.

\begin{thm}[\cite{Kuz}, Theorem 2.3, \cite{IP}]
If $X$ is a Fano threefold with Picard rank $1$ and index $2$, then $1 \leq d \leq 5$. For each $d$, the deformation class is unique and there is the following explicit description:
\begin{itemize}
\item if $d=5$, then $X=\emph{Gr}(2,5) \cap \P^6 \subset \P^9$;
\item if $d=4$, then $X$ is an intersection of two $4$-dimensional quadrics in $\P^5$;
\item if $d=3$, then $X$ is a cubic hypersurface in $\P^4$;
\item if $d=2$, then $X \to \P^3$ is a double covering ramified in a quartic;
\item if $d=1$, then $X$ is a hypersurface of degree $6$ in the weighted projective space $\P(3, 2, 1, 1, 1)$.
\end{itemize}
\end{thm}

By \cite[Corollary 3.5]{Kuz}, the bounded derived category of coherent sheaves on $X$ has a semiorthogonal decomposition of the form
$$
\D(X)=\langle \Ku(X), \mathcal{O}_{X}, \mathcal{O}_{X}(H)\rangle,
$$
where $\Ku(X)$ is called the \emph{Kuznetsov component} of $X$. 

\begin{rmk}
\label{rmk_d=4or5_first}
For $d=4$ or $5$ the Kuznetsov component has an explicit description. Indeed, if $d=4$, then $\Ku(X)\cong \D(C)$,
where $C$ is a smooth curve of genus $2$ (cf.\ \cite[Theorem 4.4]{Kuz}). If $d=5$, then $\Ku(X)\cong \D(K(3))$,
where $K(3)$ is the Kronecker quiver with three arrows (cf.\ \cite[Theorem 4.2]{Kuz}).
\end{rmk}

\subsection{Stability conditions on the Kuznetsov component}
In this section, we apply Proposition \ref{prop_inducstab} to induce stability conditions on $\Ku(X)$ of a Fano threefold $X$ of index $2$ from the weak stability conditions $\sigma_{\alpha,\beta}^0$ of Proposition \ref{prop_secondtiltstab}. This computation appeared in \cite[Theorem 6.8]{BLMS}. Here we explicit the values of $\alpha$ and $\beta$ for which the inducing method works. 
We set $\A(\alpha, \beta):=\Coh_{\alpha, \beta}^{0}(X)\cap \Ku(X)$ and $Z(\alpha, \beta):=Z_{\alpha, \beta}^{0}|_{\Ku(X)}$, 
where $Z_{\alpha, \beta}^{0}=-i Z_{\alpha, \beta}$. We define the lattice 
$$\Lambda_{H, \Ku(X)}^{2}:= \mbox{Im}(K(\Ku(X)) \to K(X) \to \Lambda_H^2) \cong \Z^{2}.$$

\begin{thm}[{\cite[Theorem 6.8]{BLMS}}]
\label{thm_U}
Suppose $-\frac{1}{2} \leq \beta < 0$, $0< \alpha < -\beta$, or $-1 < \beta <-\frac{1}{2}$, $\alpha \leq 1+\beta$. Then the pair
\begin{equation*}
\sigma(\alpha, \beta):=(\A(\alpha, \beta), Z(\alpha, \beta))
\end{equation*}
is a Bridgeland stability condition on $\Ku(X)$ 
with respect to $\Lambda_{H, \Ku(X)}^{2}\cong \Z^{2}$.
\end{thm}

\begin{proof}
Note that $\mathcal{O}_{X}, \mathcal{O}_{X}(H)\in \Coh(X)$; applying the Serre functor of $\D(X)$ we get
$$
S_{X}(\mathcal{O}_{X})=\mathcal{O}_{X}(K_{X})[3]=\mathcal{O}_{X}(-2H)[3]\in\Coh(X)[3]
$$
and
$$
S_{X}(\mathcal{O}_{X}(H))=\mathcal{O}_{X}(H+K_{X})[3]=\mathcal{O}_{X}(-H)[3]\in\Coh(X)[3].
$$
In the next, we check that $\mathcal{O}_{X}$, $\mathcal{O}_{X}(H)$, $\mathcal{O}_{X}(-2H)[2]$ and $\mathcal{O}_{X}(-H)[2]$ belong to the heart $\Coh_{\alpha, \beta}^{0}(X)$ under the previous assumptions.
%Therefore, it needs a weak stability condition whose heart contains not only $\mathcal{O}_{X}$ and $\mathcal{O}_{X}(H)$ but also $\mathcal{O}_{X}(-2H)[2]$ and $\mathcal{O}_{X}(-H)[2]$.
%So, $\Coh(X)$ needs to be tilted twice.

Since
\begin{eqnarray*}
\mu_{H}(\mathcal{O}_{X}) =0 &>& \beta \\
\mu_{H}(\mathcal{O}_{X}(-2H))=-2 &<& \beta \\
\mu_{H}(\mathcal{O}_{X}(H))=1 &>& \beta \\
\mu_{H}(\mathcal{O}_{X}(-H))=-1 &<& \beta,
\end{eqnarray*}
and these are $\mu_H$-stable line bundles, it follows that $\mathcal{O}_{X}, \mathcal{O}_{X}(-2H)[1], \mathcal{O}_{X}(H),\mathcal{O}_{X}(-H)[1]\in \Coh^{\beta}(X)$. Then, by \cite[Corollary 3.11 (a)]{BMS}, 
$\mathcal{O}_{X}$, $\mathcal{O}_{X}(-2H)[1]$, $\mathcal{O}_{X}(H)$ and $\mathcal{O}_{X}(-H)[1]$ 
are $\sigma_{\alpha, \beta}$-stable for $\alpha>0$.

Note the following inequalities:
\begin{eqnarray*}
\Re Z_{\alpha, \beta}(\mathcal{O}_{X}) 
&=& \frac{1}{2}\alpha^{2}H^{3}-\frac{\beta^{2}}{2}H^{3}<0,\\
\Re Z_{\alpha, \beta}(\mathcal{O}_{X}(-2H)[1]) 
&=& -\frac{1}{2}\alpha^{2}H^{3}+\frac{(\beta+2)^{2}}{2}H^{3}>0, \\
\Re Z_{\alpha, \beta}(\mathcal{O}_{X}(H)) 
&=& \frac{1}{2}\alpha^{2}H^{3}-\frac{(1-\beta)^{2}}{2}H^{3}<0, \\
\Re Z_{\alpha, \beta}(\mathcal{O}_{X}(-H)[1]) 
&=&  -\frac{1}{2}\alpha^{2}H^{3}+\frac{(1+\beta)^{2}}{2}H^{3} \geq 0.
\end{eqnarray*}
As $\Im Z_{\alpha, \beta}>0$ on $\mathcal{O}_{X}$, $\mathcal{O}_{X}(-2H)[1]$, $\mathcal{O}_{X}(H)$ and $\mathcal{O}_{X}(-H)[1]$, 
it follows that
$$
\mu_{\alpha, \beta}(\mathcal{O}_{X})>0,\; \mu_{\alpha, \beta}(\mathcal{O}_{X}(H))>0
$$
and
$$
\mu_{\alpha, \beta}(\mathcal{O}_{X}(-2H)[1])<0,\; \mu_{\alpha,\beta}(\mathcal{O}_{X}(-H)[1]) \leq 0.
$$

As a consequence, we obtain 
$\mathcal{O}_{X}, \mathcal{O}_{X}(-2H)[2], \mathcal{O}_{X}(H), \mathcal{O}_{X}(-H)[2]\in \Coh_{\alpha, \beta}^{0}(X)$. 
By \cite[Lemma 2.16]{BLMS}, there are no nonzero objects $F\in \Coh_{\alpha, \beta}^{0}(X) \cap \Ku(X)$ with $Z_{\alpha, \beta}^{0}(F)=0$.
The claim follows from Proposition \ref{prop_inducstab}.
\end{proof}

\begin{rmk}
Note that \cite[Theorem 6.8]{BLMS} and Theorem \ref{thm_U} hold for Fano threefolds of index two without the assumption on the Picard rank.
\end{rmk}

\begin{rmk}
\label{rmk_planeofstability}
In \cite[Section 1]{LiZhao} the authors introduce an alternative way to visualize the weak stability conditions $\sigma_{\alpha,\beta}$, which will be used in some computations of the last section. More precisely, a complex $E \in \D(X)$ such that $\ch_{\leq 2}(E) \neq (0,0,0)$ is represented by the point $[H^3\ch_0(E):H^2\ch_1(E):H\ch_2(E)]$ in a projective space, and when $\rk(E) \neq 0$, by the affine coordinates 
$$( s(E):=\frac{H^2\ch_1(E)}{H^3 \ch_0(E)}, q(E):=\frac{H\ch_2(E)}{H^3\ch_0(E)} ) \in \mathbb{A}^2_{\R}.$$
For every $(s,q) \in \mathbb{A}^2_{\R}$ with $q > \frac{1}{2} s^2$, the pair $\sigma'_{q,s}=(\Coh^s(X), Z'_{q,s})$, where
$$Z_{q,s}^{\prime}(E):=-\left( H\ch_2(E) -q H^3\ch_0(E)  \right)+ \sqrt{-1}\left( H^2\ch_1(E)-sH^3\ch_0(E) \right),$$
is a weak stability condition on $\D(X)$ with respect to $\Lambda_H^2$, as 
$$\sigma_{\alpha,\beta}= \sigma_{ \frac{\beta^2+\alpha^2}{2},\beta}^{\prime} \quad \text{and} \quad \mu_{\alpha,\beta}=\mu'_{\frac{\beta^2+\alpha^2}{2}, \beta}-\beta.$$
Thus by Theorem \ref{thm_U} the weak stability conditions $\sigma'_{q,s}$ which after tilting at $\mu'_{q,s}=-\frac{1}{2}$ restrict to stability conditions on $\Ku(X)$ are in the area
$$\{ (s,q) \in \mathbb{A}^2_{\R}: -\frac{1}{2} \leq s <0, \frac{1}{2}s^2 < q < s^2 \text{ or } -1 < s <-\frac{1}{2}, \frac{1}{2}s^2 < q \leq s^2+s+\frac{1}{2} \}.$$
Note that we will use the notation $ (\sigma'_{q,s})^{-\frac{1}{2}}$ for the weak stability condition obtained by tilting $\sigma'_{q,s}$ at $\mu'_{q,s}=-\frac{1}{2}$. Also for a point $P=(s(P),q(P))$ over the parabola $q-\frac{1}{2}s^2=0$, we will use the notation $\sigma_P:=\sigma'_{q(P),s(P)}$.
\end{rmk}

\subsection{Orbit of $\sigma(\alpha,\beta)$}
Let $X$ be a Fano threefold with $\text{Pic}(X) \cong \Z$ and of index $2$ with polarization $H:=-\frac{1}{2}K_X$. We set
\begin{equation*}
V:= \lbrace (\alpha,\beta) \in \R_{>0} \times \R : -\frac{1}{2} \leq \beta < 0, \alpha < -\beta, \text{ or } -1 < \beta <-\frac{1}{2}, \alpha \leq 1+\beta \rbrace.   
\end{equation*}
By Theorem \ref{thm_U}, we have a map 
$$
\varphi: V \to \Stab(\Ku(X))
$$
defined by 
$$
(\alpha,\beta) \mapsto \sigma(\alpha,\beta)=(\AA(\alpha,\beta),Z(\alpha,\beta)).
$$
Using slicings (see Definition \ref{def_slicing}), we write $\sigma(\alpha,\beta)=(\PP(\alpha,\beta),Z(\alpha,\beta))$, 
with
$$
\PP(\alpha,\beta)((0,1])=\AA(\alpha,\beta).
$$ 
Note that if $\PP^0_{\alpha,\beta}$ is the slicing in $\D(X)$ such that $\PP^0_{\alpha,\beta}((0,1])=\Coh^0_{\alpha,\beta}(X)$, then
$$\PP^0_{\alpha,\beta}(\phi) \cap \Ku(X) \subset \PP(\alpha,\beta)(\phi).$$
We will also denote by $\PP_{\alpha,\beta}$ the slicing such that $\PP_{\alpha,\beta}((0,1])=\Coh^\beta(X)$. 

The main result of this section states that the stability conditions in $\varphi(V)$ are in the same orbit of a fixed one $\sigma(\alpha_0,-\frac{1}{2})$ with respect to the $\tilde{\mathrm{GL}}_2^+(\R)$-action.

\begin{prop}
\label{prop_conncomp}
Fix $0< \alpha_0 < \frac{1}{2}$. For every $(\alpha,\beta) \in V$, there is $\tilde{g} \in \tilde{\emph{GL}}^+_2(\R)$ such that $\sigma(\alpha,\beta)= \sigma(\alpha_0,-\frac{1}{2}) \cdot \tilde{g}$.
\end{prop}

We need the following starting lemmas.

\begin{lemma}
\label{lemma_tiltCoh}
If $-\frac{1}{2} < \beta < 0$, then $\CCoh^\beta(X)$ is a tilt of $\CCoh^{-\frac{1}{2}}(X)$. If $-1 < \beta <-\frac{1}{2}$, then $\CCoh^\beta(X)$ is a tilt of $\CCoh^{-\frac{1}{2}}(X)[-1]$.
\end{lemma}
\begin{proof}
Consider the case $-\frac{1}{2} < \beta < 0$. Let $F$ be a $\mu_H$-semistable object in $\Coh(X)$. If $\mu_H(F)>\beta>-\frac{1}{2}$, then $F \in \Coh^\beta(X) \cap \Coh^{-\frac{1}{2}}(X)$. If $\mu_H(F) \leq \beta$, then $F[1] \in \Coh^\beta(X)$. We have the following possibilities. If $-\frac{1}{2} < \mu_H(F) \leq \beta$, then $F \in \Coh^{-\frac{1}{2}}(X)$, and thus $F[1] \in \Coh^{-\frac{1}{2}}(X)[1]$. Otherwise, if $\mu_H(F) \leq -\frac{1}{2}$, then $F[1] \in \Coh^{-\frac{1}{2}}(X)$. As every element in $\Coh^\beta(X)$ is an extension of semistable coherent sheaves as above, we conclude that 
$$\Coh^\beta(X) \subset \langle \Coh^{-\frac{1}{2}}(X), \Coh^{-\frac{1}{2}}(X)[1] \rangle.$$
This implies the first part of the statement by \cite[Exercise 6.5]{MS}.

Now assume $-1 < \beta <-\frac{1}{2}$. Let $F$ be a $\mu_H$-semistable object in $\Coh(X)$. If $\mu_H(F)> \beta$, then $F \in \Coh^{\beta}(X)$. When $\mu_H(F)> -\frac{1}{2}$, we have $F \in \Coh^{-\frac{1}{2}}(X)$, while if $\beta < \mu_H(F) \leq -\frac{1}{2}$, then $F[1] \in \Coh^{-\frac{1}{2}}(X)$, i.e.\ $F \in \Coh^{-\frac{1}{2}}(X)[-1]$.  Otherwise, if $\mu_H(F) \leq \beta < -\frac{1}{2}$, then $F[1] \in \Coh^\beta(X) \cap \Coh^{-\frac{1}{2}}(X)$. It follows that 
$$\Coh^\beta (X) \subset \langle \Coh^{-\frac{1}{2}}(X)[-1], \Coh^{-\frac{1}{2}}(X) \rangle.$$
This ends the proof of the statement.
\end{proof}

As a consequence, we get the following relation between the hearts on $\Ku(X)$.
\begin{lemma}
\label{prop_tilt}
Fix $0< \alpha_0 < \frac{1}{2}$. If $-\frac{1}{2} < \beta < 0$ and $(\alpha,\beta) \in V$, then $\AA(\alpha,\beta)$ is a tilt of $\AA(\alpha_0,-\frac{1}{2})$. If $-1 < \beta <-\frac{1}{2}$ and $(\alpha,\beta) \in V$, then $\AA(\alpha,\beta)$ is a tilt of $\AA(\alpha_0,-\frac{1}{2})[-1]$. If $\beta=-\frac{1}{2}$ and $(\alpha,\beta) \in V$, then $\AA(\alpha,\beta)=\AA(\alpha_0,\beta)$.
\end{lemma}
\begin{proof}
Firstly, we observe that $\Coh^\beta(X)=\PP^0_{\alpha,\beta}((-\frac{1}{2}, \frac{1}{2}])$ for every $\alpha >0$ up to objects supported on points. Indeed, an object $F \in \D(X)$ is $\sigma_{\alpha,\beta}$-semistable if and only if it is $\sigma^0_{\alpha,\beta}$-semistable up to objects supported on points (see \cite[Proof of Proposition 2.15]{BLMS}). This is a consequence of the fact that we are tilting a weak stability condition, so the stability is preserved up to objects with vanishing central charge. Consider a $\sigma_{\alpha,\beta}$-semistable object $F \in \Coh^\beta(X)$. Then $\Re(Z^0_{\alpha,\beta}(F))=\Im(Z_{\alpha,\beta}(F)) \geq 0$. Note that, if $\Im(Z_{\alpha,\beta}(F))=0$, then $\Im(Z^0_{\alpha,\beta}(F))=-\Re(Z_{\alpha,\beta}(F)) \geq 0$. This implies that $F \in \PP^0_{\alpha,\beta}(\frac{1}{2})$. Assume $\Im(Z_{\alpha,\beta}(F))>0$. If $\mu_{\alpha,\beta}(F)>0$, then $F \in \Coh^0_{\alpha,\beta}(X)$. It follows that $F \in \PP^0_{\alpha,\beta}((0,\frac{1}{2}))$. On the other hand, if $\mu_{\alpha,\beta}(F) \leq 0$, then $F \in \Coh^0_{\alpha,\beta}(X)[-1]$. Then, we have $F \in \PP^0_{\alpha,\beta}((-\frac{1}{2},0])$.  We deduce that $\Coh^\beta(X) \subset \PP^0_{\alpha,\beta}((-\frac{1}{2}, \frac{1}{2}])$. Since they are both hearts, we conclude that they are the same up to objects supported on points, as we claimed.

Now, assume $-\frac{1}{2} < \beta < 0$. As a consequence, by Lemma \ref{lemma_tiltCoh} we have
$$\Coh^\beta(X)=\PP^0_{\alpha,\beta}((-\frac{1}{2}, \frac{1}{2}]) \subset \PP^0_{\alpha_0,-\frac{1}{2}}((-\frac{1}{2}, \frac{3}{2}])=\langle \Coh^{-\frac{1}{2}}(X),\Coh^{-\frac{1}{2}}(X)[1] \rangle,$$
up to objects supported on points.
The same relation holds after rotating by $\frac{\pi}{2}$, namely
\begin{equation}
\label{eq_1}  
\Coh^0_{\alpha,\beta}(X)=\PP^0_{\alpha,\beta}((0, 1]) \subset \PP^0_{\alpha_0,-\frac{1}{2}}((0,2])= \langle \Coh^0_{\alpha_0,-\frac{1}{2}}(X), \Coh^0_{\alpha_0,-\frac{1}{2}}(X)[1] \rangle,
\end{equation}
up to objects supported on points. An analogous relation holds when $-1 < \beta <-\frac{1}{2}$. If $\beta=-\frac{1}{2}$, then by Lemma  \ref{lemma_tiltCoh}, we have
\begin{equation}
\label{eq_2} 
\Coh^{-\frac{1}{2}}(X)= \PP^0_{\alpha,\beta}((-\frac{1}{2}, \frac{1}{2}])=\PP^0_{\alpha_0,\beta}((-\frac{1}{2}, \frac{1}{2}]),
\end{equation}
up to objects supported on points.

Finally, we restrict to the heart $\AA(\alpha,\beta)$. By construction, the cohomology with respect to $\AA(\alpha,\beta)$ of an object $F \in \Ku(X)$ is the same as the cohomology with respect to $\Coh^0_{\alpha,\beta}(X)$ (see \cite[Lemma 4.3]{BLMS}). Thus, the statement is a consequence of \eqref{eq_1} or \eqref{eq_2}, and the fact that $\Ku(X)$ does not contain objects supported on points.
\end{proof}

The next lemma implies that the central charges $Z(\alpha,\beta)$ for $(\alpha,\beta) \in V$ are in the same orbit by the action of $\mathrm{GL}^+_2(\R)$. Recall that by \cite[Proposition 3.9]{Kuz}), a basis for the numerical Grothendieck group $\NN(\Ku(X))$ of $\Ku(X)$ is
$$\kappa_1:=[\II_\ell]=1-\frac{H^2}{d}, \kappa_2:=H-\frac{H^2}{2}+\frac{d-6}{6d}H^3.$$
\begin{lemma} 
\label{lemma_orient}
For every $(\alpha,\beta) \in V$, the image of the stability function $Z(\alpha,\beta)$ is not contained in a line and the basis $\lbrace Z(\alpha,\beta)(\kappa_1), Z(\alpha,\beta)(\kappa_2) \rbrace$ of $\C$ have the same orientation.
\end{lemma}
\begin{proof}
The matrix
$$
\begin{pmatrix} 
-\beta & 1\\
\frac{\beta^2}{2}-\frac{1}{d}-\frac{1}{2}\alpha^2 & -\beta-\frac{1}{2}
\end{pmatrix}
$$
has positive determinant for every $\beta$.
Thus the basis $Z(\alpha,\beta)(\kappa_1)$, $Z(\alpha,\beta)(\kappa_2)$ have the same orientation for every $\alpha, \beta \in \R$ with respect to the standard basis of $\C$.
\end{proof}

\begin{proof}[Proof of Proposition \ref{prop_conncomp}]
Fix $(\alpha, \beta) \in V$. By Lemma \ref{lemma_orient}, there is an element $\tilde{g} \in \tilde{\mathrm{GL}}^+_2(\R)$ such that $\sigma(\alpha_0,-\frac{1}{2}) \cdot \tilde{g}=(\AA',Z(\alpha,\beta))$. By definition, denoting $\tilde{g}=(g,M)$ with $g:\R \to \R$ and $M \in \mathrm{GL}^+_2(\R)$, we have
$$\AA'=\PP'((0,1])=\PP(\alpha_0,-\frac{1}{2})((r,r+1]),$$ 
where $r:=g(0)$. Up to shifting $\AA'$ by an integer, we can assume that $\AA'$ is a tilt of $\AA(\alpha_0,-\frac{1}{2})$.

On the other hand, by Lemma \ref{prop_tilt} the heart $\AA(\alpha,\beta)$ is a tilt of $\AA(\alpha_0,-1/2)$ up to shift. Since $\sigma(\alpha_0,-\frac{1}{2}) \cdot \tilde{g}$ and $\sigma(\alpha,\beta)$ have the same stability function and their hearts are tilt of the same heart, they are the same stability condition by \cite[Lemma 8.11]{BMS}. This proves the statement.
\end{proof}

\begin{rmk}
Proposition \ref{prop_conncomp} implies that the map $\varphi$ is continuous.
\end{rmk}

Fix a stability condition $\sigma(\alpha_0,-\frac{1}{2})$ with $0 < \alpha_0 < \frac{1}{2}$. We set 
$$\KK:=\sigma(\alpha_0,-\frac{1}{2}) \cdot \tilde{\mathrm{GL}}^+_2(\R) \subset \Stab(\Ku(X)),$$ 
which is the universal covering space of $M^+:= Z(\alpha_0,-\frac{1}{2}) \cdot \mathrm{GL}^+_2(\R)$. By Proposition \ref{prop_conncomp}, the image $\varphi(V)$ is contained in $\KK$, and by definition $\KK$ is an open subset of a connected component of $\Stab(\Ku(X))$. 

\begin{rmk}
\label{rmk_almostcc}
Note that all elements in $\KK$ satisfy the support property with respect to the trivial quadratic form $Q=0$, because their central charge is injective. Recall the setting of \cite[Proposition A.5]{BMS}. Consider the open subset of $\Hom(\NN(\Ku(X)),\C)$ containing central charges which are not injective. Let $U$ be the connected component containing $Z(\alpha_0,-\frac{1}{2})$ of this open subset. Let $\mathcal{U}$ be the connected component of the preimage $\mathcal{Z}^{-1}(U)$ containing $\sigma(\alpha_0,-\frac{1}{2})$. We have that $U=M^+$. Indeed, note that $\Hom(\NN(\Ku(X)),\C)$ is the disjoint union of $M^+$ with the component of matrices with negative determinant and the component $M^0$ containing matrices with determinant equal to $0$. Since central charges parametrized by $M^0$ have non-trivial kernel, we get the above equality. By \cite[Proposition A.5]{BMS}, the restriction $\mathcal{Z}|_{\mathcal{U}}: \mathcal{U} \to M^+$ is a covering map. As $\KK$ is the universal covering space of $M^+$, there is a covering map $\KK \to \mathcal{U}$ commuting with $\mathcal{Z}$. But $\KK$ is a subset of $\mathcal{U}$, so we conclude that $\KK=\mathcal{U}$. In particular, we deduce that $\KK$ is a connected component of $\mathcal{Z}^{-1}(U)$.
\end{rmk}

\begin{rmk}
\label{rmk_d=4or5}
If $d=4$, then $\Ku(X)\cong \D(C)$,
where $C$ is a smooth curve of genus $2$ (cf.\ \cite[Theorem 4.4]{Kuz}). By \cite[Theorem 2.7]{Macri}, we have $\KK\cong\tilde{\mathrm{GL}}^+_2(\R) \cong \Stab(\Ku(X))$.

If $d=5$, then $\Ku(X)\cong \D(K(3))$,
where $K(3)$ is the Kronecker quiver with three arrows (cf.\ \cite[Theorem 4.2]{Kuz}). In this case, by \cite[Theorem 1.1]{DK19}, it is known that $\Stab(\Ku(X)) \cong \HH \times \C$, where $\HH:=\lbrace z \in \C: \Im z >0 \rbrace$.
\end{rmk}

\section{Hilbert scheme of lines and stability}

In this section, we study the stability of ideal sheaves of lines in $X$ and of their dual object; then we prove Theorem \ref{Fanolines_modspace}.

\subsection{Lines and stability}

Let $X$ be a Fano threefold of Picard rank $1$ and index $2$. Given a line $\ell \subset X$, we denote by $\II_\ell$ the ideal sheaf of $\ell$ in $X$. By \cite[Proposition 3.12]{Kuz2}, we know that $\II_\ell \in \Ku(X)$. The Chern character of $\II_{\ell}$ is 
$$\ch(\II_{\ell})=(1,0,-\frac{1}{d}H^{2}, 0)$$ and the twisted Chern character with respect to $-\frac{1}{2}$ till degree $2$ is
$$\ch^{-\frac{1}{2}}_{\leq 2}(\II_{\ell})=(1,\frac{1}{2}H,\frac{d-8}{8d}H^{2}).$$

\begin{prop}
\label{lemma_idealsheafstab}
The ideal sheaf $\II_\ell$ of a line $\ell \subset X$ is $\sigma(\alpha,-\frac{1}{2})$-stable for every $0 < \alpha < \frac{1}{2}$. 
\end{prop}
\begin{proof}
The sheaf $\II_{\ell}$ is  slope stable, because it is a torsion-free sheaf of rank $1$. As $\mu_{H}(\II_{\ell})=0>-\frac{1}{2}$, we have $\II_{\ell}\in \Coh^{-\frac{1}{2}}(X)$. Since $H^{2}\ch_{1}^{-\frac{1}{2}}(\II_{\ell})>0$, it follows from \cite[Lemma 2.7]{BMS} that $\II_{\ell}$ is $\sigma_{\alpha, -\frac{1}{2}}$-stable for any $\alpha\gg 0$.

In the next we show that there are no walls for the stabilty of $\II_\ell$ with respecto to $\sigma_{\alpha,-\frac{1}{2}}$. A wall would be given by a short exact sequence in the heart $\Coh^{-\frac{1}{2}}(X)$ of the form
$$
0\rightarrow E\rightarrow \II_{\ell}\rightarrow F \rightarrow 0,
$$
such that the following conditions hold:
\begin{enumerate}
\item[(i)] $\mu_{\alpha, -\frac{1}{2}}(E)=\mu_{\alpha, -\frac{1}{2}}(\II_{\ell})=\mu_{\alpha, -\frac{1}{2}}(F)$;
\item[(ii)] $\Delta_H(E)\geq 0$, $\Delta_H(F)\geq 0$;
\item[(iii)]  $\Delta_H(E)\leq \Delta_H(\II_{\ell})$, $\Delta_H(E)\leq \Delta_H(\II_{\ell})$.
\end{enumerate}
The truncated twisted characters of $E$ and $F$ have to satisfy
\begin{equation*}
(1,\frac{1}{2}H,\frac{d-8}{8d}H^{2})
=(a,\frac{b}{2}H, \frac{c}{8d}H^{2})+(1-a,\frac{1-b}{2}H,\frac{d-8-c}{8d}H^{2}),
\end{equation*}
for some $a, b, c \in \Z$. As $E$ and $F$ are in $\Coh^{-\frac{1}{2}}(X)$, we have
$$b \geq 0 \quad \text{and} \quad 1-b \geq 0,$$
i.e.\ $b=0$ or $b=1$.

As $\mu_{\alpha, -\frac{1}{2}}(\II_{\ell})=\frac{d-8-4d\alpha^{2}}{4d}$ 
and $\Delta_H(\II_{\ell})=\frac{2}{d}(H^3)^2$, dividing the discriminant by $d^2$, the previous conditions are 
\begin{enumerate}
\item[(i)] $\frac{1}{b}(\frac{c}{4d}-\alpha^{2}a)=\frac{d-8-4d\alpha^{2}}{4d}=\frac{1}{1-b}(\frac{d-8-c}{4d}-\alpha^{2}(1-a))$;
\item[(ii)] $(\frac{b}{2})^{2}-\frac{ac}{4d}\geq 0$, $(\frac{1-b}{2})^{2}+\frac{(1-a)(c+8-d)}{4d}\geq 0$;
\item[(iii)]  $(\frac{b}{2})^{2}-\frac{ac}{4d}\leq \frac{2}{d}$, $(\frac{1-b}{2})^{2}+\frac{(1-a)(c+8-d)}{4d}\leq \frac{2}{d}$.
\end{enumerate}

Assume $b=0$. If $a \neq 0$, we get 
$$4d \alpha^2=\frac{c}{a}>0 \quad \text{and} \quad -8 \leq ac \leq 0$$
which is impossible. If $a=0$, then $c=0$ by the first equation. If $E$ has twisted character $(1,\frac{1}{2}H,\frac{d-8}{8d}H^{2})$, then $\II_\ell$ has a subobject with the same slope for every $\alpha >0$. This contradicts the stability of $\II_\ell$ above the wall. In the other case, $\II_\ell$ would have a subobject with infinite slope, in contradiction with the fact that it is stable for $\alpha$ large. 

If $b=1$ and $a \neq 1$, we get a contradiction from
$$4d \alpha^2=\frac{8-d+c}{a-1}>0 \quad \text{and} \quad -8 \leq (a-1)(8-d+c) \leq 0.$$
The case $b=a=1$ can be excluded as done for $b=a=0$.

Since there are no values of $a, b, c$ satisfying all the required conditions, we deduce that $\II_{\ell}$ is $\sigma_{\alpha, -\frac{1}{2}}$-stable for any $\alpha>0$. As a consequence, $\II_\ell$ is $\sigma_{\alpha,-\frac{1}{2}}^0$-stable and thus stable with respect to $\sigma(\alpha,-\frac{1}{2})$.
\end{proof}

As explained in \cite[Section 3.6]{Kuz2}, there is another object in $\Ku(X)$ naturally associated to a line $\ell \subset X$. Indeed, consider the triangle
\begin{equation}
\label{eq_defJ}   
\OO_X(-1)[1] \to \JJ_\ell \to \OO_\ell(-1),
\end{equation}
where $\OO_X(1):=\OO_X(H)$.
By \cite[Lemma 3.4]{Kuz2}, we have $\JJ_\ell \in \Ku(X)$. The Chern character of $\JJ_\ell$ is
$$\ch(\JJ_\ell)=(-1,H,\frac{2-d}{2d}H^2,\frac{d-6}{6d}H^3).$$
We study the stability of this object, which will be used in the proof of Theorem \ref{thm_conncomp}.

\begin{lemma}
\label{lemma_stabJontop}
Set $\beta=-\frac{1}{2}$. The complex $\JJ_\ell$ is $\sigma_{\alpha,\beta}$-stable for $\alpha \gg 0$.
\end{lemma}
\begin{proof}
Note that $\OO_X(-1)[1]$ and $\OO_\ell(-1)$ belong to the heart $\Coh^\beta(X)$. Thus the same property holds for their extension $\JJ_\ell$.

Assume that $\JJ_\ell$ is not stable. Then there exists a destabilizing sequence in $\Coh^\beta(X)$ of the form
$$0 \to P \to \JJ_\ell \to Q \to 0,$$
where $P$ and $Q$ are $\mu_{\alpha,\beta}$-semistable for $\alpha \gg 0$, with $\mu_{\alpha,\beta}(P) > \mu_{\alpha,\beta}(Q)$. Recall that we can check the stability with rispect to $\sigma_{\alpha,\beta}$ for $\alpha \to \infty$ using the slope $\frac{\ch_1^{\beta}(-)}{\rk(-)}$, as
$$\lim_{\alpha \to \infty} \frac{2}{\alpha^2}\frac{\ch_2^{\beta}(-)-\frac{1}{2}\alpha^2\rk(-)}{\ch_1^\beta(-)}=-\frac{\rk(-)}{\ch_1^\beta(-)}$$
and taking the opposite of the inverse does not change the inequalities.

Consider the cohomology sequence
$$0 \to \HH^{-1}(P) \to \HH^{-1}(\JJ_\ell) \to \HH^{-1}(Q) \to \HH^0(P) \to \HH^0(\JJ_\ell) \to \HH^0(Q) \to 0,$$
where $\HH^{-1}(\JJ_\ell)=\OO_X(-1)$ and $\HH^0(\JJ_\ell)=\OO_\ell(-1)$ by definition. Then $\rk(\HH^{-1}(P))=1$ as it is a subsheaf of a line bundle, and $\rk(\HH^0(Q))=0$ as it is the quotient of a torsion sheaf. Since $\rk(\HH^0(P)) \geq 0$, we have $\rk(P) \geq -1$.

If $\rk(P) \geq 1$, then $\rk(Q) \leq -2$ as $\rk(P)+\rk(Q)=\rk(\JJ_\ell)=-1$. Since $Q$ is a destabilizing quotient for $\JJ_\ell$, we have the relation
$$\frac{\ch_1^{-\frac{1}{2}}(\JJ_\ell)}{\rk(\JJ_\ell)}= -\frac{1}{2} \geq \frac{\ch_1^{-\frac{1}{2}}(Q)}{\rk(Q)} \geq -\frac{\ch_1^{-\frac{1}{2}}(Q)}{2}$$
which implies $\ch_1^{-\frac{1}{2}}(Q) \geq 1$. This contradicts the fact that $0 \leq \ch_1^{-\frac{1}{2}}(Q) \leq \ch_1^{-\frac{1}{2}}(\JJ_\ell)=\frac{1}{2}$. It follows that the rank of $P$ can be equal to $0$ or $-1$. 

Assume that $P$ has rank $-1$. Then $\HH^{-1}(Q)=0$, because $Q$ is a torsion object in the heart. Thus we have the sequence 
$$0 \to \HH^0(P) \to \OO_\ell(-1) \to Q \to 0,$$
where $\HH^0(P)$ is a sheaf supported on the line $\ell$. Note that $\OO_\ell(-1)$ has rank $1$ and it is torsion free as a sheaf on $\ell$. It follows that $\HH^0(P)$ is a rank $1$ torsion free sheaf on $\ell$. As a consequence, we have 
$\ch_{\leq 2}(\HH^0(P))=\ch_{\leq 2}(\OO_\ell(-1))$, which implies 
$$\ch_{\leq 2}(P)=\ch_{\leq 2}(\JJ_\ell) \quad \text{and} \quad \mu_{\alpha,\beta}(Q)=+\infty \text{ for }\alpha \gg 0.$$
This is impossible for a destabilizing sequence.

Assume now that $P$ has rank $0$. If $\ch_1(P) \neq 0$, then $\mu_{\alpha,\beta}(P)$ would be a finite number, while $\mu_{\alpha,\beta}(Q)=+\infty$ for $\alpha \gg 0$, which is impossible. If $\ch_1(P)=\rk(P)=0$, then we have $P \cong \HH^0(P)$, $\HH^{-1}(Q) \cong \OO_X(-1)$ and the sequence
$$0 \to P \to \OO_\ell(-1) \to \HH^0(Q) \to 0.$$
As $\OO_\ell(-1)$ is locally free on $\ell$ and $P$ is a subsheaf of $\OO_\ell(-1)$, we must have $\ch_2(P)=\ch_2(\OO_\ell(-1))$; thus $\HH^0(Q)$ is supported in codimension $3$. On the other hand, consider the commutative diagram
$$
\xymatrix{
0 \ar[r] \ar[d]& P \ar[r]^\cong \ar[d]& P \ar[r] \ar[d]& \ar[d]0\\
\OO_X(-1)[1] \ar[r] \ar[d]^\cong & \JJ_\ell \ar[r] \ar[d] & \OO_\ell(-1) \ar[r]^{\psi} \ar[d]^{\phi} & \OO_X(-1)[2] \ar[d]^\cong\\
\OO_X(-1)[1] \ar[r] & Q \ar[r]& \HH^0(Q) \ar[r]^\gamma& \OO_X(-1)[2],
}
$$
where $\psi, \phi \neq 0$. Note that $\Hom(\HH^0(Q),\OO_X(-1)[2])=0$, as $\HH^0(Q)$ is supported on points and by Serre duality. It follows that $\gamma =0$, which is impossible. This proves $\JJ_\ell$ is stable as claimed.
\end{proof}

\begin{prop}
\label{prop_stabJ}
The complex $\JJ_\ell$ is $\sigma(\alpha,-\frac{1}{2})$-stable for every $0 < \alpha < \frac{1}{2}$.
\end{prop}
\begin{proof}
Set $\beta=-\frac{1}{2}$. By Lemma \ref{lemma_stabJontop}, it is enough to check that there are no walls for the stability of $\JJ_\ell$ with respect to $\sigma_{\alpha,\beta}$. As in the case of the ideal sheaf $\II_\ell$, this is enough to prove the statement.

Note that
$$\ch_{\leq 2}^\beta(\JJ_\ell)=(-1, \frac{1}{2}H, \frac{8-d}{8d}H^2)$$
and
$$\mu_{\alpha,\beta}(\JJ_\ell)=\frac{8-d+4d\alpha^2}{4d}, \quad \Delta_H(\JJ_\ell)=\frac{2}{d}(H^3)^2.$$
Assume there is a short exact sequence in the heart $\Coh^{\beta}(X)$ of the form
$$
0\rightarrow E\rightarrow \JJ_{\ell}\rightarrow F \rightarrow 0,
$$
corresponding to a wall for $\JJ_\ell$. The twisted characters of $E$ and $F$ have to satisfy
\begin{equation*}
(-1,\frac{1}{2}H,\frac{8-d}{8d}H^{2})
=(a,\frac{b}{2}H, \frac{c}{8d}H^{2})+(-1-a,\frac{1-b}{2}H,\frac{8-d-c}{8d}H^{2}),
\end{equation*}
for some $a, b, c \in \Z$. As $E$ and $F$ are in $\Coh^\beta(X)$, we have $b=0$ or $b=1$.

Assume $b=0$. From the equality of slopes and the bounds on $\Delta_H$, we get
$$c=4d\alpha^2 \quad \text{and} \quad -\frac{8d}{3} \leq ac \leq 0.$$
This is impossible, unless $a=0$. But in this case, we have $c=0$ and $\JJ_\ell$ would have either a subobject with infinite slope or a subobject with the same slope. Both possibilities would give a contradiction with Lemma \ref{lemma_stabJontop}. The case $b=1$ can be excluded similarly.
\end{proof}

\subsection{Proof of Theorem \ref{Fanolines_modspace}}
Assume again that $X$ is a Fano threefold of Picard rank one and index two and of degree $1\leq d\leq 5$. We denote by $\Sigma(X)$ the Hilbert scheme parametrizing lines in $X$. Let us summarize what is known on $\Sigma(X)$ (see also \cite[Section 5]{CZ}):
\begin{itemize}
%\item For a line $\ell \subset X$, the normal bundle of $\ell$ in $X$ is $\NN_\ell \cong \OO_\ell(a) \oplus \OO_\ell(-a)$ for $a \geq 0$. Then $\Sigma(X)$ is smooth of dimension $2$ at points corresponding to lines with $a=0,1$, and singular at points with $a \geq 2$ (see \cite[Corollary 2.1.6]{KPS}).
\item If $d=1$, then $\Sigma(X)$ is a projective and irreducible scheme given by a smooth surface with an embedded curve, whose reduced scheme is smooth (see \cite[Theorem 4]{Tihomirov}).
\item If $d \geq 2$, then $\Sigma(X)$ is $2$-dimensional and generically smooth (see \cite[Lemma 2.2.6]{KPS}).  
\item If $d \geq 3$, then $\Sigma(X)$ is a smooth and irreducible surface. In particular:
\begin{itemize}
\item if $d=3$, then $\Sigma(X)$ is a minimal surface of general type;
\item if $d=4$, then $\Sigma(X)$ is an abelian surface;
\item if $d=5$, then $\Sigma(X) \cong \P^2$;
\end{itemize}
(see \cite[Proposition 2.2.10]{KPS}).
\end{itemize}

Making use of the results in the previous section, we describe $\Sigma(X)$ as a moduli space of objects in $\Ku(X)$.

\begin{prop}
\label{lemma_idealsheafstabinK}
The ideal sheaf $\II_\ell$ of a line $\ell \subset X$ is $\sigma$-stable for every $\sigma \in \KK$. 
\end{prop}
\begin{proof}
Note that it is enough to check the stability with respect to $\sigma(\alpha,-\frac{1}{2})$ for a certain $0< \alpha < \frac{1}{2}$, as $\tilde{\mathrm{GL}}^+_2(\R)$-action preserves the stability. This is provided by Proposition \ref{lemma_idealsheafstab}.
\end{proof}

\begin{rmk}
Similarly, Proposition \ref{prop_stabJ} imply that $\JJ_\ell$ is $\sigma$-stable for every $\sigma \in \KK$.
\end{rmk}

\begin{prop}\label{lem_numclassI}
Assume $d \neq 1$. If $F \in \Ku(X)$ is $\sigma$-stable for $\sigma \in \KK$ with $[F]=[\II_{\ell}]\in \mathcal{N}(\Ku(X))$, then $F\cong \II_{\ell'}[2k]$ for some line $\ell'\subset X$ and $k \in \Z$.
\end{prop}

\begin{proof}
Set $\ch(F):=(a_0, a_1H, a_2H^2, a_3H^3)$. 
As $[F]=[\II_{\ell}]\in \mathcal{N}(\Ku(X))$ and $\chi(\II_{\ell},\II_{\ell})=-1$, the following conditions hold:
$$
\chi(\mathcal{O}_{X}, F)=0,\; \chi(\mathcal{O}_{X}(1), F)=0,\; \chi(\II_{\ell}, F)=-1, \chi(F, \II_{\ell})=-1.
$$
Recall that
$$
\td(X)=(1, H, (\frac{1}{3}+\frac{1}{d})H^{2}, \frac{1}{d}H^{3})
$$
(cf.\ case 3 of \cite[Lemma 1.2]{Li}). By Hirzebruch-Riemann-Roch Theorem, we get 
$$
\begin{cases}
a_0+ \frac{d+3}{3}a_1+da_2+da_3=0  \\
\frac{6-d}{6}a_1+da_3=0  \\
\frac{d}{3}a_1+da_2+da_3=-1  \\
-\frac{d}{3}a_1+da_2-da_3=-1.
\end{cases}
$$
If $d \neq 2$, then $a_0=1, a_1=0, a_2=-\frac{1}{d}, a_3=0$. If $d=2$, the condition $\chi(F, F)=-1$ implies $\ch(F)=\ch(\II_\ell)$.

Now by assumption and Proposition \ref{prop_conncomp}, $F$ is $\sigma(\alpha,\beta)$-stable for every $(\alpha,\beta) \in V$. In particular, $F[2k+1] \in \AA(\alpha, \beta)$ for some integer $k$. Up to shifting, we may assume $G:=F[1] \in \AA(\alpha, \beta)$ is $\sigma(\alpha,\beta)$-stable with slope $$\mu_{\alpha,\beta}^0(G)=\frac{-\beta}{\frac{2-d\beta^2}{2d}+\frac{1}{2}\alpha^2}.$$
In particular, we have $\mu^0_{\alpha,\beta}(G)=+\infty$ if
\begin{equation}
\label{eq_slopeinfinity}    
2-d\beta^2+d\alpha^2=0.
\end{equation}

We distinguish the cases $d \geq 3$ and $d=2$. If $d \geq 3$ we can find pairs $(\bar{\alpha},\bar{\beta}) \in V$ such that \eqref{eq_slopeinfinity} holds (in this case $\bar{\beta}$ satisfies $-\frac{d+2}{2d}< \bar{\beta}< -\sqrt{\frac{2}{d}}$). Then $G$ is $\sigma^0_{\bar{\alpha},\bar{\beta}}$-semistable, since $G$ has the largest slope in the heart. A similar computation as in Proposition \ref{lemma_idealsheafstab} shows that $\beta=-1$ is not on a wall for the $\sigma_{\alpha,\beta}^0$-stability of $G$. Moreover, the semicircle $\CC$ in the $(\alpha,\beta)$-plane of center $(0,-\frac{d+2}{2d})$ and ray $\frac{d-2}{2d}$ gives a numerical wall for $G$, which could be realized for instance by the object $\OO_X(-1)[2] \in \Coh^0_{\alpha,\beta}(X)$. 

Assume that $\CC$ is not an actual wall for $G$. All the other numerical walls would be nested semicircles in $\CC$.
Thus we may choose $(\bar{\alpha},\bar{\beta})=(\frac{d-2}{2d},-\frac{d+2}{2d})$, so that $G$ is $\sigma^0_{\bar{\alpha},\bar{\beta}}$-semistable and $G$ remains semistable for $\bar{\beta}$ approaching $-\frac{1}{2}$ (see Figure \ref{fig0}).
\begin{figure}[htb]
\centering
\begin{tikzpicture}[domain=-8:1]
\draw[->] (-8,0) -- (1,0) node[right] {$\beta$};
\draw[->] (0,-2) -- (0, 5.5) node[above] {$\alpha$};
\coordinate (q1) at (0,0);

\coordinate (q2) at (0,4);
\node  at (q2) {$-$};
\draw (0, 4) node [right]  {$\frac{1}{2}$};
\coordinate (s1) at (-4,0);
\node  at (s1) {$|$};
\draw (-4,0) node [below]  {$-\frac{1}{2}$};
\coordinate (s2) at (-8,0);
\node  at (s2) {$\mid$};
\draw (-8,0) node [below]  {$-1$};

\coordinate (OH) at (-4,4);
%\node  at (OH) {$\bigcdot$};

\coordinate (A) at (-6.66,0);
\node  at (A) {$\bigcdot$};
\draw (-6.66,0) node [below]{$-\frac{d+2}{2d}$};

\draw[dotted] (s2)--(OH)  (0.05, -0.005);
\draw[dotted] (q1)--(OH)  (0.05, -0.005);

\draw[dashed] (-5.33,0) arc (0:180:1.33cm);
\draw (-5.66,0) arc (0:180:1cm);

\coordinate (P) at (-6.66, 1.33);
\node  at (P) {$\bigcdot$};
\draw (-6.66,1.33) node [above right]{$(\bar{\alpha},\bar{\beta})$};

\end{tikzpicture}
\caption{The point $(\bar{\alpha},\bar{\beta})$ lies above the first wall for $G$. \label{fig0}}
\end{figure}
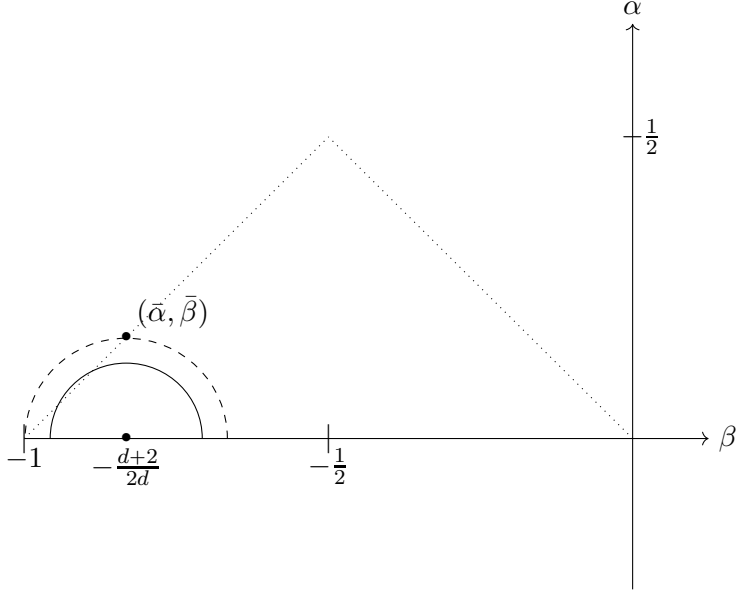

Now we argue as in \cite[Lemma 2.15]{PeRo}. By definition of $\Coh^0_{\alpha,-\frac{1}{2}}(X)$, we have a triangle
$$A[1] \to G \to B$$
such that $A$ (resp.\ $B$) is in $\Coh^{-\frac{1}{2}}(X)$ with $\sigma_{\alpha,-\frac{1}{2}}$-semistable factors having slope $\mu_{\alpha,-\frac{1}{2}} \leq 0$ (resp.\ $> 0$). Since $G$ is $\sigma^0_{\alpha,-\frac{1}{2}}$-semistable, we have that $B$ is either supported on points or $0$. Moreover, $A[1]$ is $\sigma^0_{\alpha,-\frac{1}{2}}$-semistable and, since $A \in \Coh^{-\frac{1}{2}}(X)$, we have that $A$ is $\sigma_{\alpha,-\frac{1}{2}}$-semistable. Hence $\ch(A)=(1, 0, -\frac{1}{d}H^2, m)$, where $m \geq 0$ is the length of the support of $B$. The walls computation in Proposition \ref{lemma_idealsheafstab} shows that $A$ is $\sigma_{\alpha,-\frac{1}{2}}$-semistable for every $\alpha>0$. Hence by \cite{Li}, \cite[Conjecture 4.1]{BMS} holds for $\alpha=0, \beta=-\frac{1}{2}$, so that
$$4 ( \frac{d-8}{8} )^2-6 \frac{d}{2}\ch_3^{-\frac{1}{2}}(A) \geq 0.$$
Performing the computation, we get that $\ch_3(A)=m <1$ for $d \neq 1$. We deduce that $m=0$, so $B=0$. We conclude that $G=A[1]$, equivalently $F=A$ is $\sigma_{\alpha,-\frac{1}{2}}$-semistable for every $\alpha>0$. By \cite[Lemma 2.7]{BMS},  it follows that $F$ is a slope semistable torsion free sheaf.  As $\ch(F)=\ch(\II_{\ell})$ and $\mathrm{Pic}(X)\cong \Z$, we must have $F\cong \II_{\ell'}$ for some line $\ell'\in X$. 

Assume instead that $\CC$ defines an actual wall for $G$ and that $G$ becomes unstable for $\beta \to -\frac{1}{2}$. Set $(\bar{\alpha},\bar{\beta})=(\frac{d-2}{2d},-\frac{d+2}{2d})$. Then $G$ is strictly $\sigma^0_{\bar{\alpha},\bar{\beta}}$-semistable and there is a sequence
$$0 \to P \to G \to Q \to 0$$
in $\Coh^0_{\bar{\alpha},\bar{\beta}}(X)$ where $P, Q$ are $\mu^0_{\bar{\alpha},\bar{\beta}}$-semistable with the same slope $+\infty$. A computation shows that $\ch_{\leq 2}(P)=(1,-H, \frac{1}{2}H^2)$, $\ch_{\leq 2}(Q)=(-2,H, \frac{2-d}{2d}H^2)$ and $d \neq 5$. 

Note that $P \cong \OO_X(-1)[2]$. Indeed, we have $\Hom(\OO_X(-1)[2], P[i])=0$ for every $i \neq 0,1$, by Serre duality and the fact that they are in $\Coh^0_{\bar{\alpha},\bar{\beta}}(X)$. Denote by $\HH^i(P)$ the degree-$i$ cohomology of $P$ in $\Coh^{\bar{\beta}}(X)$. Then $\HH^0(P)$ is either $0$ or supported on points. Then 
$$\Hom(\OO_X(-1)[2], P[1])=\Hom(\OO_X(-1)[1],P)=\Hom(\OO_X(-1),\HH^{-1}(P))=0,$$
since $\OO_X(-1)$ is stable with the same character of $\HH^{-1}(P)$ till degree $2$. As $\chi(\OO_X(-1),P) \neq 0$, we deduce that $\Hom(\OO_X(-2)[2], P) \neq 0$. The stability of $\OO_X(-2)[2]$ implies the claim. 

As a consequence, $\chi(P,Q) \neq 0$ and $\Hom(P,Q[i])=0$ for $i \neq 1$, since $P$ and $Q$ are in the same heart and by Serre duality. So we can define $G'$ as the extension
$$0 \to Q \to G' \to P \to 0$$
in $\Coh^0_{\bar{\alpha},\bar{\beta}}(X)$. Now $G'$ is $\sigma^0_{\alpha,\beta}$-semistable for $\beta \to -\frac{1}{2}$. But then the computation done in the previous case for $\beta=-\frac{1}{2}$ shows that $G' \cong \II_\ell[1]$. However, this would imply $G' \in \Ku(X)$, so that $\Hom(G',P)=0$ giving a contradiction. This proves the statement for $d >2$.

If $d=2$, note that \eqref{eq_slopeinfinity}  holds for $(\bar{\alpha},\bar{\beta})=(0,-1)$. Since there are no walls intersecting the vertical line $\beta=-1$, we deduce that $G$ is $\sigma_{\alpha,\beta}^0$-semistable for $\beta=-1$ and $\alpha>0$. As explained before, $G$ sits in a sequence
$$A[1] \to G \to B$$
where $B$ is either $0$ or supported on points and $A$ is $\sigma_{\alpha,-1}$-semistable in $\Coh^{-1}(X)$. Applying \cite[Conjecture 4.1]{BMS} for $\alpha=0$, we deduce $G=A[1]$. The argument applied in the previous case proves the statement for $d=2$.
\end{proof}

\begin{proof}[Proof of Theorem \ref{Fanolines_modspace}]
Assume $d \neq 1$. By Proposition \ref{lemma_idealsheafstabinK} and Proposition \ref{lem_numclassI} there is a bijection between $\Sigma(X)$ and $M_\sigma(\Ku(X),[\II_\ell])$. Moreover, the universal ideal sheaf on $X \times \Sigma(X)$ is a universal family of stable objects on $X \times M_\sigma(\Ku(X),[\II_\ell])$. Thus the bijection induces an isomorphism, arguing as in \cite[Section 5.2]{BMMS}. The case $d=1$ follows from Proposition \ref{lemma_idealsheafstabinK} by a similar argument as above. 
\end{proof}

\begin{rmk}
By \cite[Section 3.6]{Kuz2} the complex $\JJ_\ell$ is the (derived) dual of $\II_\ell(1)$ shifted by $1$. In fact $\II_\ell$ and $\JJ_\ell$ are exchanged by the autoequivalence 
$$\Phi(-):=\mathsf{R}\lHom(-,\OO_X(-1))[1]: \Ku(X) \to \Ku(X).$$
As a consequence, we deduce that $\Sigma(X) \cong M_\sigma(\Ku(X),[\II_\ell]) \cong M_{\Phi \cdot \sigma}(\Ku(X),[\JJ_\ell])$ for $d \neq 1$.
\end{rmk}

\begin{rmk}
Note that there exist a non-trivial map $\II_\ell \to \JJ_{\ell'}$ if $\ell$ and $\ell'$ intersect in a point and a non-zero morphism $\JJ_\ell \to \II_{\ell'}[1]$ for disjoint $\ell$ and $\ell'$. As consequence, if $\II_\ell$ and $\JJ_\ell$ are $\sigma$-stable of phases $\phi_\II$ and $\phi_\JJ$, respectively, for a stability condition $\sigma$, then
\begin{equation}
\label{eq_phases}   
\phi_\JJ-1 < \phi_\II < \phi_\JJ.
\end{equation}
\end{rmk}

\begin{rmk}
Assume $X$ is a cubic threefold ($d=3$) and let $\bar{\sigma}$ be the stability condition on $\Ku(X)$ constructed in \cite{BMMS}. By \cite[Theorem 4.1]{BMMS} the ideal sheaves of lines are $\bar{\sigma}$-stable. Moreover, since $\Hom(\JJ_\ell,\JJ_\ell[i])= \Hom(\II_\ell,\II_\ell[i])$, by \cite[Proposition 4.2]{BMMS}, it follows that $\JJ_\ell$ is $\bar{\sigma}$-stable. An interesting question would be to understand whether $\bar{\sigma}$ belongs to $\KK$.
\end{rmk}

\section{Cubic threefolds}
Assume that $X$ is a cubic threefold. In this section we show that the Serre functor of $\Ku(X)$ preserves the orbit $\KK$. As a consequence, we deduce Theorem \ref{cor_smoothmod} and we give an alternative proof of the categorical Torelli Theorem proved in \cite{BMMS}.

\subsection{Serre functor and stability}
Let $X$ be a cubic threefold. Recall that $\Ku(X)$ is a Calabi-Yau category of dimension $5/3$, which means that $S_{\Ku(X)}^3 \cong [5]$ by \cite[Lemmas 4.1 and 4.2]{Kuz04}. Moreover, if we denote by $\Phi: \D(X) \to \D(X)$ the autoequivalence $\Phi(-)=(-) \otimes \OO_X(H)$, by \cite[Lemma 4.1]{Kuz04} the Serre functor of $\Ku(X)$ satisfies the relation
\begin{equation}
\label{eq_Serrefunctor}
S_{\Ku(X)}^{-1}=(\L_{\OO_X} \circ \Phi) \circ (\L_{\OO_X} \circ \Phi)[-3].
\end{equation}
Here $\L_{\OO_X}: \D(X) \to \D(X)$ is the left mutation functor with respect to $\OO_X$. Let us firstly study the action of the autoequivalence $\L_{\OO_X} \circ \Phi$ of $\Ku(X)$ on a stability condition $\sigma(\alpha,-\frac{1}{2})$. We need the following lemmas.

\begin{lemma}
\label{lemma_firsttilt}
The heart $\Phi(\emph{Coh}^{-\frac{1}{2}}(X))$ of $\emph{D}^b(X)$ is a tilt of $\emph{Coh}^{-\frac{1}{2}}(X)$.
\end{lemma}
\begin{proof}
Let $F \in \Coh(X)$ be a slope semistable sheaf with $\mu_H(F) > -\frac{1}{2}$. Then $\Phi(F)$ is a slope semistable sheaf with $\mu_H(\Phi(F))> \frac{1}{2}> -\frac{1}{2}$. Thus $\Phi(F) \in \Coh^{-\frac{1}{2}}(X)$. If $F \in \Coh(X)$ is slope semistable with $\mu_H(F) \leq -\frac{1}{2}$, then we can have either $-\frac{1}{2} < \mu_H(\Phi(F)) \leq \frac{1}{2}$, or $\mu_H(\Phi(F)) \leq - \frac{1}{2}$. In the first case, $\Phi(F)[1] \in \Coh^{-\frac{1}{2}}(X)[1]$, in the second we have $\Phi(F)[1] \in \Coh^{-\frac{1}{2}}(X)$. We deduce $\Phi(\Coh^{-\frac{1}{2}}(X)) \subset \langle \Coh^{-\frac{1}{2}}(X), \Coh^{-\frac{1}{2}}(X)[1] \rangle$, as we wanted.
\end{proof}

\begin{lemma}
\label{lemma_secondtilt}
Assume that $E \in \emph{Coh}^0_{\alpha,-\frac{1}{2}}(X)$ for every $0 < \alpha < \frac{1}{2}$. Then there exists $\varepsilon>0$ such that for $\alpha':=\frac{1}{2}-\varepsilon>0$ we have
$$\Phi(E) \in \langle \emph{Coh}_{\alpha',-\frac{1}{2}}^0(X), \emph{Coh}_{\alpha',-\frac{1}{2}}^0(X)[1], \OO_X[2] \rangle.$$
\end{lemma}
\begin{proof}
Set $k:=\text{dim}\Hom(\OO_X(-H)[2],E)$ and consider the exact triangle
\begin{equation}
\label{eq_extwithO-H}
\OO_X(-H)^{\oplus k}[2] \to E \to E_1
\end{equation}
with $E_1$ satisfying $\Hom(\OO_X(-H)[2],E_1)=0$. Denote by $ \Coh^{-\frac{1}{2}}(X)_{\mu_{\alpha, -\frac{1}{2}}  > 0}$ (resp.\ $ \Coh^{-\frac{1}{2}}(X)_{\mu_{\alpha, -\frac{1}{2}}  \leq  0}$) the subcategory of $\Coh^{-\frac{1}{2}}(X)$ generated by $\mu_{\alpha,-\frac{1}{2}}$-semistable objects with slope $\mu_{\alpha,-\frac{1}{2}} >0$ (resp.\ $\leq 0$). Let us summarize the argument of the proof. Firstly we note that up to choosing a $\alpha$ close to $\frac{1}{2}$, we may assume that $\OO_X(-H)^{\oplus k}[2]$ is the $\sigma_{\alpha,-\frac{1}{2}}$-semistable factor of $E$ with bigger slope. In particular, $E_1$ belongs to $\Coh^0_{\alpha,-\frac{1}{2}}(X)$. Thus there exist $A \in \Coh^{-\frac{1}{2}}(X)_{\mu_{\alpha, -\frac{1}{2}}  \leq  0}$, $B \in  \Coh^{-\frac{1}{2}}(X)_{\mu_{\alpha, -\frac{1}{2}}  > 0}$ and an extension
$$A[1] \to E_1 \to B$$
in $\Coh^0_{\alpha,-\frac{1}{2}}(X)$. We will show that, choosing a suitable $\alpha'=\frac{1}{2}-\varepsilon >0$ with $\varepsilon>0$, we have
\begin{equation}
\label{eq_phi(B)}
\Phi(B) \in \langle \Coh^{-\frac{1}{2}}(X)_{\mu_{\alpha', -\frac{1}{2}} > 0}, \Coh^{-\frac{1}{2}}(X)[1] \rangle
\end{equation}
and
\begin{equation}
\label{eq_phi(A)}
\Phi(A)[1] \in \langle \Coh^{-\frac{1}{2}}(X)[1], \Coh^{-\frac{1}{2}}(X)_{\mu_{\alpha', -\frac{1}{2}} \leq 0}[2] \rangle.
\end{equation}
Together with \eqref{eq_extwithO-H}, this implies the statement as
$$\langle \Coh^{-\frac{1}{2}}(X)_{\mu_{\alpha', -\frac{1}{2}} > 0}, \Coh^{-\frac{1}{2}}(X)[1] \rangle \subset \langle \Coh^0_{\alpha',-\frac{1}{2}}(X), \Coh^0_{\alpha',-\frac{1}{2}}(X)[1] \rangle$$
and
$$\langle \Coh^{-\frac{1}{2}}(X)[1], \Coh^{-\frac{1}{2}}(X)_{\mu_{\alpha', -\frac{1}{2}} \leq 0}[2] \rangle \subset \langle \Coh^0_{\alpha',-\frac{1}{2}}(X), \Coh^0_{\alpha',-\frac{1}{2}}(X)[1]  \rangle.$$
We point out that $E$ belongs to the heart after changing $\alpha$ by hypothesis.

In order to do the computation, it is convenient to use the setup introduced in \cite[Section 1]{LiZhao}, recalled in Remark \ref{rmk_planeofstability}. Then we will prove the statement using the weak stability condition $\sigma_{q,s}'$ with reparamentrized central charge, and the condition $\alpha'=\frac{1}{2} - \varepsilon$ is equivalent to $q=\frac{1}{4}-\varepsilon$.

Firstly, we claim that, up to changing $q$, we can assume that $\OO_X(-H)^{\oplus k}[2]$ is the $(\sigma_{q,-\frac{1}{2}}^{\prime})^{ -\frac{1}{2}}$-semistable factor of $E$ with bigger slope. Indeed, note that $$(\mu_{\frac{1}{4},-\frac{1}{2}}')^{-\frac{1}{2}}(\OO_X(-H)[2])= + \infty;$$ 
thus the slope $(\mu_{q,-\frac{1}{2}}')^{-\frac{1}{2}}(\OO_X(-H)[2])$ converges to $+\infty$ for $q \to \frac{1}{4}$. Now assume that $E$ has a $(\sigma_{q,-\frac{1}{2}}')^{-\frac{1}{2}}$-semistable factor $A_i$ with $(\mu_{q,-\frac{1}{2}}')^{-\frac{1}{2}}(A_i) \geq (\mu_{q,-\frac{1}{2}}')^{-\frac{1}{2}}(\OO_X(-H)[2])$. Then we can choose $\varepsilon>0$ such that the slope $(\mu_{\frac{1}{4}-\varepsilon,-\frac{1}{2}}')^{-\frac{1}{2}}(A_i) < (\mu_{\frac{1}{4}-\varepsilon,-\frac{1}{2}}')^{-\frac{1}{2}}(\OO_X(-H)[2])$ as required (see Figure \ref{fig_1}).

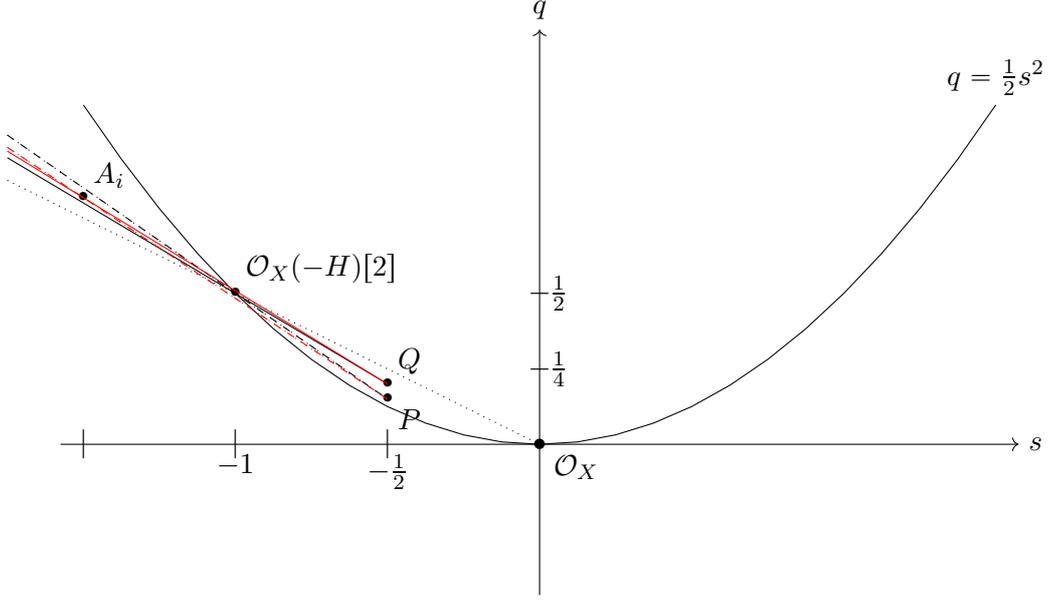
\begin{figure}[htb]
\centering
\begin{tikzpicture}[domain=-6:6]
\draw[->] (-6.3,0) -- (6.3,0) node[right] {$s$};
\draw[->] (0,-2) -- (0, 5.5) node[above] {$q$};
\draw plot (\x,{0.125*\x*\x}) node [above] {$q = \frac{1}{2}s^{2}$};
\coordinate (O) at (0,0);
\node  at (O) {$\bullet$};

\coordinate (q1) at (0,1);
\node  at (q1) {$-$};
\draw (0, 1) node [right]  {$\frac{1}{4}$};
\coordinate (q2) at (0,2);
\node  at (q2) {$-$};
\draw (0, 2) node [right]  {$\frac{1}{2}$};
\coordinate (s1) at (-2,0);
\node  at (s1) {$|$};
\draw (-2,0) node [below]  {$-\frac{1}{2}$};
\coordinate (s2) at (-4,0);
\node  at (s2) {$\mid$};
\draw (-4,0) node [below]  {$-1$};
\coordinate (s3) at (-6,0);
\node  at (s3) {$\mid$};

\coordinate (OH) at (-4,2);
\node  at (OH) {$\bigcdot$};
\draw (-4,2) node [above right]  {$\mathcal{O}_{X}(-H)[2]$};

\coordinate (A) at (-6,3.27);
\node  at (A) {$\bigcdot$};
\draw (-6,3.27) node [above right]{$A_{i}$};

\coordinate (Q) at (-2, 0.8);
\node  at (Q) {$\bigcdot$};
\coordinate (P) at (-2, 0.6);
\node  at (P) {$\bigcdot$};

\coordinate (OHOE) at (-7, 3.5);
\coordinate (OHQE) at (-7, 3.8);
\coordinate (OHPE) at (-7, 4.1);
\coordinate (AQE) at (-7, 3.8875);
\coordinate (APE) at (-7, 3.9375);

\draw[dotted] (OHOE)--(O)  (0.05, -0.005)  node [below right]   {$\OO_X$};

%slop O(-H)[2] w.r.s Q&P 
\draw (OHQE)--(Q) node [above right]{$Q$}; 
\draw[dashed] (OHPE)--(P)  node [below right]{$P$};
\draw[dotted] (OHPE)--(P); 

%slop A_{i} w.r.s Q&P 
\draw[red] (AQE)--(Q) ;
\draw[red, dashed] (APE)--(P);
\draw[red, dotted] (APE)--(P);

\end{tikzpicture}
\caption{In this picture we compare the slope of $\OO_X(-H)[2]$ and $A_i$.  The slope with respect to $P$ is determined by the dotted-dashed lines, while that with respect to $Q$ by full lines. We use red lines for the slope of $A_i$ and black lines for that of $\OO_X(-H)$. Note that $A_i$ has bigger slope at $P$ than $\OO_X(-H)[2]$, while for $Q$ approaching $(-\frac{1}{2}, \frac{1}{4})$ the slope of $A_i$ with respect to $Q$ is less than that of $\OO_X(-H)[2]$. \label{fig_1}}
\end{figure}
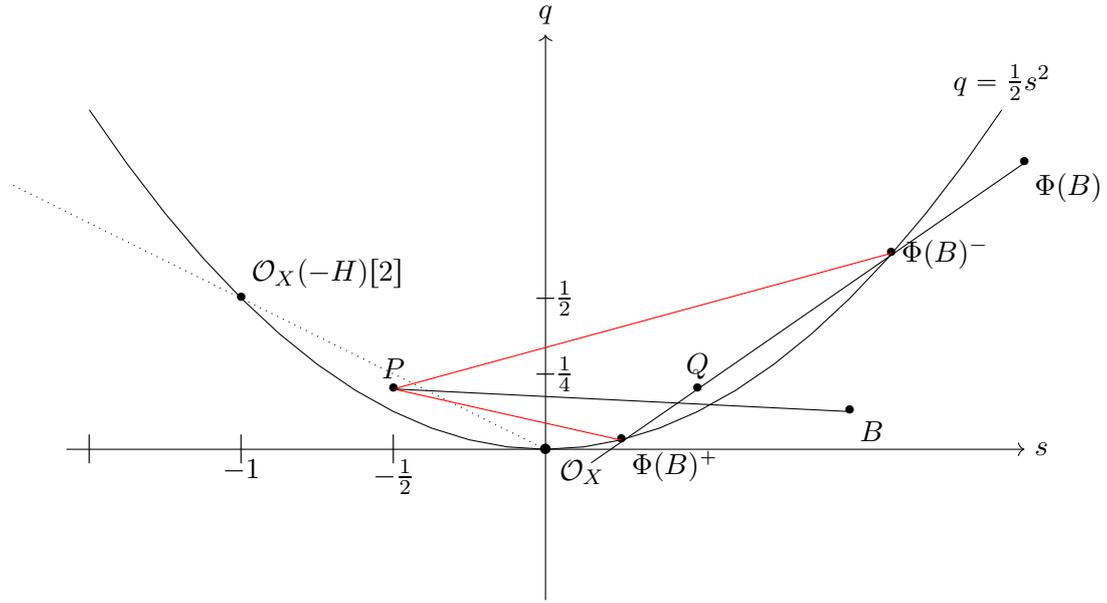
\begin{figure}[htb]
\centering
\begin{tikzpicture}[domain=-6:6]
\draw[->] (-6.3,0) -- (6.3,0) node[right] {$s$};
\draw[->] (0,-2) -- (0, 5.5) node[above] {$q$};
\draw plot (\x,{0.125*\x*\x}) node [above] {$q = \frac{1}{2}s^{2}$};
\coordinate (O) at (0,0);
\node  at (O) {$\bullet$};

\coordinate (q1) at (0,1);
\node  at (q1) {$-$};
\draw (0, 1) node [right]  {$\frac{1}{4}$};
\coordinate (q2) at (0,2);
\node  at (q2) {$-$};
\draw (0, 2) node [right]  {$\frac{1}{2}$};
\coordinate (s1) at (-2,0);
\node  at (s1) {$|$};
\draw (-2,0) node [below]  {$-\frac{1}{2}$};
\coordinate (s2) at (-4,0);
\node  at (s2) {$\mid$};
\draw (-4,0) node [below]  {$-1$};
\coordinate (s3) at (-6,0);
\node  at (s3) {$\mid$};

\coordinate (OH) at (-4,2);
\node  at (OH) {$\bigcdot$};
\draw (-4,2) node [above right]  {$\mathcal{O}_{X}(-H)[2]$};

\coordinate (Q) at (2, 0.8);
\node  at (Q) {$\bigcdot$};
\draw (2,0.8) node [above] {$Q$};

\coordinate (P) at (-2, 0.8);
\node  at (P) {$\bigcdot$};
\draw (-2,0.8) node [above] {$P$};

\coordinate (OHOE) at (-7, 3.5);
\draw[dotted] (OHOE)--(O)  (0.05, -0.005)  node [below right]   {$\OO_X$};

\coordinate (PBJA) at (4.55, 2.6);
\node  at (PBJA) {$\bigcdot$};
\draw (4.55, 2.6) node [right]  {$\Phi(B)^{-}$};
\draw[red] (P)--(PBJA);

\coordinate (PBJAB) at (6.3, 3.798);
\node  at (PBJAB) {$\bigcdot$};

\coordinate (PBJ) at (1, 0.12);
\node  at (PBJ) {$\bigcdot$};
\draw (1, 0.12) node [below right]  {$\Phi(B)^{+}$};

\draw[red] (P)--(PBJ) ;

\coordinate (B) at (4, 0.5);
\node  at (B) {$\bigcdot$};
\draw (P)--(B) node [below right]  {$B$};

\coordinate (sq4) at (0.6, -0.18);
\draw (sq4)--(PBJAB) node [below right]  {$\Phi(B)$};

\end{tikzpicture}
\caption{The slope with respect to $P$ of the semistable factors of $\Phi(B)$ is between that of $\Phi(B)^-$ and $\Phi(B)^+$. \label{fig_2}} 
\end{figure}

Secondly, we show \eqref{eq_phi(B)}, i.e.\ that
$$\Phi(B) \in \langle \Coh^{-\frac{1}{2}}(X)_{\mu'_{\frac{1}{4}-\varepsilon, - \frac{1}{2}}>-\frac{1}{2}}, \Coh^{-\frac{1}{2}}(X)[1] \rangle.$$
By Lemma \ref{lemma_firsttilt} we know that $\Phi(B)$ has semistable factors in $\Coh^{-\frac{1}{2}}(X)$ and $\Coh^{-\frac{1}{2}}(X)[1]$. Up to passing to a stable factor, we may assume that $B$ is $\sigma'_{\frac{1}{4}-\varepsilon,-\frac{1}{2}}$-stable. Then $\Phi(B)$ is $\sigma'_{\frac{1}{4}-\varepsilon,\frac{1}{2}}$-stable. We can use \cite[Lemma 3]{LiZhao2} (see also \cite[Lemma 3.1]{FLLQ}) to control the slope of the $\sigma'_{\frac{1}{4}-\varepsilon,-\frac{1}{2}}$-semistable factors of $\Phi(B)$. Indeed, let $\sigma_R$ be a stability condition corresponding to a point $R$ on the segment connecting $P=(-\frac{1}{2},\frac{1}{4}-\varepsilon)$ and $Q=(\frac{1}{2},\frac{1}{4}-\varepsilon)$. The point $\Phi(B)$ has coordinate $s(\Phi(B))=s(B)+1$ and belongs to the parabola $q-\frac{1}{2}s^2=q(B)-\frac{1}{2}s(B)^2$. Denote by $\Phi(B)^+$ and $\Phi(B)^-$ the points where the parabola $q-\frac{1}{2}s^2=0$ and the line connecting $\Phi(B)$ and $Q$ intersect. By \cite[Lemma 3]{LiZhao2} the slope of a $\sigma_R$-semistable factor $F$ of $\Phi(B)$ satisfies $\mu_R(\Phi(B)^-) \leq \mu_R(F) \leq \mu_R(\Phi(B)^+)$ (see Figure \ref{fig_2}). Since $\mu_P(\Phi(B)^-) \geq \mu_P(B)>-\frac{1}{2}$, we deduce that the $\sigma_P$-semistable factors of $\Phi(B)$ are either in $\Coh^{-\frac{1}{2}}(X)$ with slope $\mu_P>-\frac{1}{2}$ or in $\Coh^{-\frac{1}{2}}(X)[1]$, as we claimed.

Finally, we study $\Phi(A)[1]$. By Lemma \ref{lemma_firsttilt} the semistable factors of $\Phi(A)[1]$ are in $\Coh^{-\frac{1}{2}}(X)[1]$ and $\Coh^{-\frac{1}{2}}(X)[2]$. Up to considering a stable factor, we may assume that $A$ is $\sigma'_{\frac{1}{4}-\varepsilon,-\frac{1}{2}}$-stable. Note that by definition $A[1]$ determines a point between the parabola $q -\frac{1}{2}s^2=0$ and the line $l_P$ passing through $P$ and parallel to the line $q=-\frac{1}{2}s$. 

Assume that $A[1]$ is in the region between the line $l_{P\OO(-H)}$ connecting $P$ with the point $(-1,\frac{1}{2})$ corresponding to $\OO_X(-H)[2]$ and the parabola $q-\frac{1}{2}s^2=0$. Denote by $\Phi(A[1])^+$ and $\Phi(A[1])^-$ the intersection points of the line connecting $\Phi(A)[1]$ and $Q$. Then by \cite[Lemma 3]{LiZhao2} a $\sigma_P$-semistable factor $F$ of $\Phi(A)[1]$ satisfies $\mu_P(\Phi(A[1])^-) \leq \mu_P(F) \leq \mu_P(\Phi(A[1])^+)$. Since for such $A[1]$, the point $\Phi(A[1])^+$ is on the left of the line $l_P$, it follows that $\mu_P(F) \leq -\frac{1}{2}$ (see Figure \ref{fig_3}).
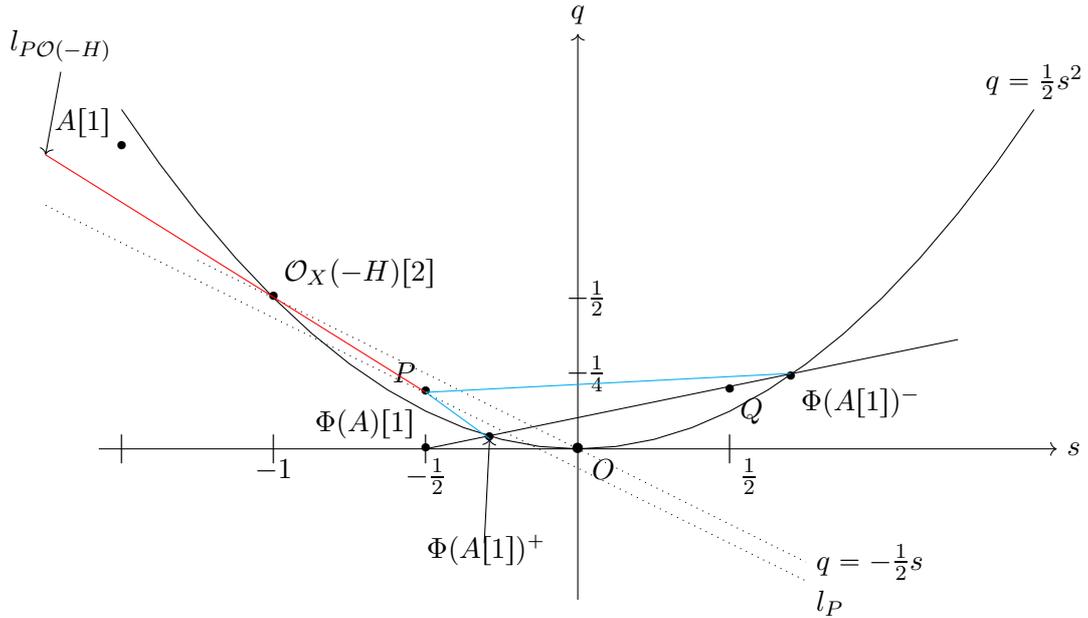
\begin{figure}[h!]
\centering
\begin{tikzpicture}[domain=-6:6]
\draw[->] (-6.3,0) -- (6.3,0) node[right] {$s$};
\draw[->] (0,-2) -- (0, 5.5) node[above] {$q$};
\draw plot (\x,{0.125*\x*\x}) node [above] {$q = \frac{1}{2}s^{2}$};
\coordinate (O) at (0,0);
\node  at (O) {$\bullet$};
\draw (0.05, -0.005)  node [below right]   {$O$};
\coordinate (q1) at (0,1);
\node  at (q1) {$-$};
\draw (0, 1) node [right]  {$\frac{1}{4}$};
\coordinate (q2) at (0,2);
\node  at (q2) {$-$};
\draw (0, 2) node [right]  {$\frac{1}{2}$};
\coordinate (s1) at (-2,0);
\node  at (s1) {$|$};
\node  at (s1) {$\bigcdot$};
\draw (-2, 0) node [above left]  {$\small{\Phi(A)[1]}$};

\coordinate (sj1) at (2,0);
\node  at (sj1) {$|$};

\draw (-2,0) node [below]  {$-\frac{1}{2}$};
\draw (2,0) node [below right]  {$\frac{1}{2}$};
\coordinate (s2) at (-4,0);
\node  at (s2) {$\mid$};
\draw (-4,0) node [below]  {$-1$};
\coordinate (s3) at (-6,0);
\node  at (s3) {$\mid$};

\coordinate (OH) at (-4,2);
\node  at (OH) {$\bigcdot$};
\draw (-4,2) node [above right]  {$\mathcal{O}_{X}(-H)[2]$};

\coordinate (Q) at (2, 0.78);
\node  at (Q) {$\bigcdot$};
\draw (2,0.8) node [below right] {$Q$};

\coordinate (P) at (-2, 0.75);
\node  at (P) {$\bigcdot$};
\draw (-2, 0.75) node [above left] {$P$};

\coordinate (OHOE) at (-5, 2.5);
\coordinate (OE) at (3, -1.5);
\draw[dotted] (OHOE)--(OE)  node [right]  {$q=-\frac{1}{2}s$};

\coordinate (PAJ) at (2.8, 0.95);
\node  at (PAJ) {$\bigcdot$};
\draw (PAJ) node [below right]  {$\Phi(A[1])^{-}$};

\coordinate (PHA) at (-1.16, 0.135);
\node  at (PHA) {$\bigcdot$};

\coordinate (OHE) at (-7, 3.9);
\draw[red] (P)--(OHE);

\coordinate (PHAJ) at (-1.23, -1.25);
\draw (-1.2, -1) node [below ]  {$\Phi(A[1])^{+}$};

\coordinate (PHAE) at (5, 1.45);

\draw (s1)--(PHAE);

\draw[->] (PHAJ)--(PHA);
\coordinate (PAJJ) at (2.8, 0.9999);
\draw[cyan] (P)--(PAJJ);
\draw[cyan] (P)--(PHA);

\coordinate (PRE) at (3, -1.75);
\coordinate (PLE) at (-7, 3.23);
\draw[dotted] (PLE)--(PRE) node [below right]  {$l_{P}$};

\coordinate (LP) at (-6.8, 5);
\draw (-6.8, 5) node [above]  {$l_{P\OO(-H)}$};
\draw[->] (LP)--(OHE);

\coordinate (AS) at (-6, 4);
\node  at (AS) {$\bigcdot$};
\draw (-6,4) node [above left] {$A[1]$};

\end{tikzpicture}
\caption{Case of $A[1]$ above the line connecting $P$ and $\OO_X(-H)[2]$. \label{fig_3}} 
\end{figure}
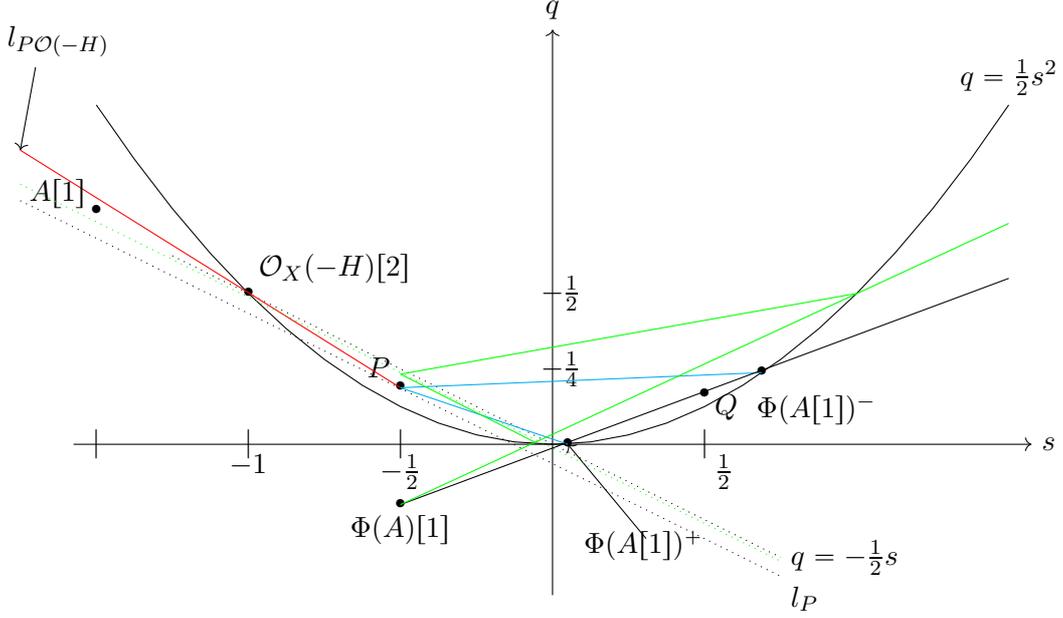
\begin{figure}[h!]
\centering
\begin{tikzpicture}[domain=-6:6]
\draw[->] (-6.3,0) -- (6.3,0) node[right] {$s$};
\draw[->] (0,-2) -- (0, 5.5) node[above] {$q$};
\draw plot (\x,{0.125*\x*\x}) node [above] {$q = \frac{1}{2}s^{2}$};

\coordinate (q1) at (0,1);
\node  at (q1) {$-$};
\draw (0, 1) node [right]  {$\frac{1}{4}$};
\coordinate (q2) at (0,2);
\node  at (q2) {$-$};
\draw (0, 2) node [right]  {$\frac{1}{2}$};
\coordinate (s1) at (-2,0);
\node  at (s1) {$|$};

\coordinate (sj1) at (2,0);
\node  at (sj1) {$|$}; 
\draw (-2,0) node [below]  {$-\frac{1}{2}$};
\draw (2,0) node [below right]  {$\frac{1}{2}$};
\coordinate (s2) at (-4,0);
\node  at (s2) {$\mid$};
\draw (-4,0) node [below]  {$-1$};
\coordinate (s3) at (-6,0);
\node  at (s3) {$\mid$};

\coordinate (OH) at (-4,2);
\node  at (OH) {$\bigcdot$};
\draw (-4,2) node [above right]  {$\mathcal{O}_{X}(-H)[2]$};

\coordinate (Q) at (2, 0.66);
\node  at (Q) {$\bigcdot$};
\draw (2,0.8) node [below right] {$Q$};

\coordinate (P) at (-2, 0.75);
\node  at (P) {$\bigcdot$};
\draw (-2, 0.75) node [above left] {$P$};

\coordinate (OHOE) at (-5, 2.5);
\coordinate (OE) at (3, -1.5);

\draw[dotted] (OHOE)--(OE)  node [right]  {$q=-\frac{1}{2}s$};

\coordinate (PAJ) at (2.75, 0.95);
\node  at (PAJ) {$\bigcdot$};
\draw (2.55, 0.8) node [below right]  {$\Phi(A[1])^{-}$};
\coordinate (PAJE) at (6, 2.2);

\draw[cyan] (P)--(PAJ);

\coordinate (PHA) at (-2, -0.8);
\node  at (PHA) {$\bigcdot$};
\draw (PHA) node [below ]  {$\Phi(A)[1]$};

\coordinate (sJ1) at (0.2,0);
\node  at (sJ1) {$\bigcdot$};

\coordinate (OHE) at (-7, 3.9);
\draw[red] (P)--(OHE);

\coordinate (PHAJ) at (1.23, -1.25);
\draw (1.2, -1) node [below ]  {$\Phi(A[1])^{+}$};
\draw[->] (PHAJ)--(sJ1);

\draw[cyan] (P)--(sJ1);

\coordinate (PRE) at (3, -1.75);
\coordinate (PLE) at (-7, 3.23);
\draw[dotted] (PLE)--(PRE) node [below right]  {$l_{P}$};

\coordinate (LP) at (-6.8, 5);
\draw (-6.5, 5) node [above]  {$l_{P\OO(-H)}$};
\draw[->] (LP)--(OHE);

\coordinate (AS) at (-6, 3.1);
\node  at (AS) {$\bigcdot$};
\draw (-6,3) node [above left] {$A[1]$};

\draw[] (PHA)--(PAJE);

\coordinate (PAB) at (-2, 0.93);
\coordinate (PAJA) at (4, 2);
\coordinate (PAJAE) at (6, 2.93);

\coordinate (PABL) at (-7, 3.45);
\coordinate (PABR) at (3, -1.55);

\draw[green,dotted] (PABL)--(PABR);

\draw[green] (PAB)--(PAJA);

\draw[green] (PHA)--(PAJAE);

\coordinate (sp1) at (-0.22, 0.022);
\draw[green] (PAB)--(sp1);

\end{tikzpicture}
\caption{Case of $A[1]$ below the line connecting $P$ and $\OO_X(-H)[2]$. Moving $P$ towards $(-\frac{1}{2},\frac{1}{4})$, the point $\Phi(A[1])^+$ moves to the left on the parabola $q-\frac{1}{2}s^2=0$, as represented by the green lines in the picture. \label{fig_4}} 
\end{figure}
This implies that $\Phi(A)[1]$ has $\sigma'_{\frac{1}{4}-\varepsilon,-\frac{1}{2}}$-semistable factors in $\Coh^{-\frac{1}{2}}(X)[1]$ and in $\Coh^{-\frac{1}{2}}(X)_{\mu'_{\frac{1}{4}-\varepsilon, - \frac{1}{2}}\leq -\frac{1}{2}}[2]$, giving \eqref{eq_phi(A)} in this case.

Assume instead $A[1]$ is in the region between the lines $l_{P\OO(-H)}$ and $l_P$. Keeping the previous notation, it might happen that $\Phi(A[1])^+$ appears on the right of $l_P$. In this case, choosing a suitable $\varepsilon>0$, we can require that $\Phi(A[1])^+$ is on the left of $l_P$ (see Figure \ref{fig_4}). Then we reduce to the previous situation and we conclude that
$$\Phi(A)[1] \in \langle \Coh^{-\frac{1}{2}}(X)[1], \Coh^{-\frac{1}{2}}(X)_{\mu'_{\frac{1}{4}-\varepsilon, - \frac{1}{2}}\leq -\frac{1}{2}}[2]  \rangle.$$
Choosing $\varepsilon$ as above, we deduce the stament for $\alpha'=\frac{1}{2}-\varepsilon$.
\end{proof}

\begin{lemma}
\label{lemma_restriction}
For $0 < \alpha < \frac{1}{2}$, the heart $\L_{\OO_X}(\Phi(\AA(\alpha,-\frac{1}{2})))$ is a tilt of $\AA(\alpha,-\frac{1}{2})$.
\end{lemma}
\begin{proof}
Note that $\L_{\OO_X}(\Phi(-))= \Phi(\L_{\OO_X(-H)}(-))$ as $\Phi$ is an equivalence. Take $E \in \AA(\alpha, -\frac{1}{2})$ and set $G:= \L_{\OO_X(-H)}(E)$. Since $\OO_X(-H)[2], \OO_X(-3H)[2]$ are in $\Coh^0_{\alpha,-\frac{1}{2}}(X)$, by Serre duality we have
$\Hom(\OO_X(-H)[2], E[i])=0$ for every $i \notin \lbrace 0, 1, 2, 3 \rbrace$. In particular, $G$ is defined by 
$$\OO_X(-H)^{\oplus k_0}[2] \oplus \OO_X(-H)^{\oplus k_1}[2][-1] \oplus \OO_X(-H)^{\oplus k_2}[2][-2] \oplus \OO_X(-H)^{\oplus k_3}[2][-3] \to E \to G.$$
The induced long exact sequence of cohomology in $\Coh^0_{\alpha,-\frac{1}{2}}(X)$ is given by
$$0 \to \HH^{-1}(G) \to \OO_X(-H)^{\oplus k_0}[2] \to E \to \HH^0(G) \to \OO_X(-H)^{\oplus k_1}[2] \to 0$$ 
and $\HH^1(G) \cong \OO_X(-H)^{\oplus k_2}[2]$, $\HH^2(G) \cong \OO_X(-H)^{\oplus k_3}[2]$.

Up to deforming $\alpha \to \frac{1}{2}$ we may assume that $\OO_X(-H)^{\oplus k_0}[2]$ is the HN factor of $E$ with maximal slope. Also $E$ satisfies the hypothesis of Lemma \ref{lemma_secondtilt}, as $E \in \AA(\alpha, -\frac{1}{2})=\AA(\alpha',-\frac{1}{2})$ for every $0 < \alpha' < \frac{1}{2}$ by Lemma \ref{prop_tilt}. So we can find $\alpha$ such that $\Phi(E) \in \langle \Coh^0_{\alpha,-\frac{1}{2}}(X), \Coh^0_{\alpha,-\frac{1}{2}}(X)[2], \OO_X[2] \rangle$. As a consequence, we have $\HH^{-1}(G)=0$; also writing
$$0 \to \OO_X(-H)^{\oplus k_0}[2] \to E \to Q \to 0,$$
\begin{equation}
\label{eq_H0(G)}
0 \to Q \to \HH^0(G) \to \OO_X(-H)^{\oplus k_1}[2] \to 0,
\end{equation}
it follows that the cohomology in $\Coh^0_{\alpha,-\frac{1}{2}}(X)$ of $\Phi(E)$ in degree $-2$ is $\OO_X^{\oplus k_0}$ and $\Phi(Q)$ has non trivial cohomology with respect to $\Coh^0_{\alpha,-\frac{1}{2}}(X)$ only in the degrees $0, -1$.

Applying $\Phi$ to \eqref{eq_H0(G)} we get
$$ \Phi(Q) \to \Phi(\HH^0(G)) \to \OO_X^{\oplus k_1}[2].$$
The long exact sequence of cohomology in $\Coh^0_{\alpha, -\frac{1}{2}}(X)$ is of the form
$$0 \to \HH^{-2}(\Phi(\HH^0(G))) \to \OO_X^{\oplus k_1} \to \HH^{-1}(\Phi(Q)) \to \HH^{-1}(\Phi(\HH^0(G))) \to 0$$
and $\HH^0(\Phi(Q))\cong \HH^0(\Phi(\HH^0(G)))$. Note that $\HH^0(G)$ satisfies the hypothesis of Lemma \ref{lemma_secondtilt}, as $Q$ and $\OO_X(-H)[2]$ do. As a consequence, $\HH^{-2}(\Phi(\HH^0(G))) \cong \OO_X^{\oplus m}$. 

On the other hand, by the previous computation of the cohomology of $G$, we can write
$$G' \to \Phi(G) \to \Phi(\HH^2(G))[-2]\cong \OO_X^{\oplus k_3}$$
and
$$ \Phi(\HH^0(G)) \to G' \to \Phi(\HH^1(G))[-1]\cong \OO_X^{\oplus k_2}[1].$$
Then we have
$$\Hom(\OO_X[2], \Phi(\HH^0(G)))= \Hom(\OO_X[2], G') = \Hom(\OO_X[2], \Phi(G))=0.$$ This contradicts the fact that $\HH^{-2}(\Phi(\HH^0(G))) \cong \OO_X^{\oplus m}$. 

As a consequence, we deduce that $\Phi(\HH^0(G)) \in \langle \Coh^0_{\alpha,-\frac{1}{2}}(X), \Coh^0_{\alpha,-\frac{1}{2}}(X)[1] \rangle$. Then $G'$ satisfies the same property and thus also $\Phi(G)$ does. In particular, we have proved that $\Phi(G)$ has non trivial cohomology in $\Coh^0_{\alpha,-\frac{1}{2}}(X)$ only in the degrees $0, -1$. Since $\Phi(G) \in \Ku(X)$, its cohomology is also in $\Ku(X)$ by \cite[Lemma 4.3]{BLMS}. We deduce that $\Phi(G) \in \langle \AA(\alpha, -\frac{1}{2}), \AA(\alpha, -\frac{1}{2})[1] \rangle$ as we wanted.
\end{proof}

We are now ready to prove that $\L_{\OO_X} \circ \Phi$ preserves the orbit $\KK$.
\begin{prop}
\label{prop_twistfctrvsstabcond}
There exists $\tilde{g} \in \tilde{\emph{GL}}^+_2(\R)$ such that 
$$(\L_{\OO_X} \circ \Phi) \cdot \sigma(\alpha,-\frac{1}{2})= \sigma(\alpha,-\frac{1}{2}) \cdot \tilde{g}.$$
\end{prop}
\begin{proof}
By definition the stability condition $(\L_{\OO_X} \circ \Phi) \cdot \sigma(\alpha,-\frac{1}{2})$ has heart $\L_{\OO_X}\Phi(\AA(\alpha,-\frac{1}{2}))$ and stability function $Z':=Z \circ (\L_{\OO_X} \circ \Phi)_*^{-1}$. On the other hand, the objects $\L_{\OO_X}\Phi(\II_\ell)$ and $\L_{\OO_X}\Phi(\JJ_\ell)$ are defined respectively by the triangles
$$\OO_X^{\oplus 3} \to \II_\ell(H) \to \L_{\OO_X}\Phi(\II_\ell) \quad \text{and} \quad \OO_X \oplus \OO_X[1] \to \JJ_\ell(H) \to \L_{\OO_X}\Phi(\JJ_\ell).$$
A standard computation shows that $(\L_{\OO_X} \circ \Phi)_*^{-1}([\II_\ell])=-[\JJ_\ell]$ and $(\L_{\OO_X} \circ \Phi)_*^{-1}([\JJ_\ell])=[\II_\ell]-[\JJ_\ell]$. Then it is possible to check that the basis $\lbrace Z(\alpha,-\frac{1}{2})(-[\JJ_\ell]), Z(\alpha,-\frac{1}{2})([\II_\ell]-[\JJ_\ell]) \rbrace$ and $\lbrace Z(\alpha,-\frac{1}{2})([\II_\ell]), Z(\alpha,-\frac{1}{2})([\JJ_\ell]) \rbrace$ have the same orientation. Thus there exists $\tilde{g} \in \tilde{\text{GL}}^+_2(\R)$ such that $\sigma(\alpha,-\frac{1}{2}) \cdot \tilde{g} =\sigma'$, where $\sigma'=(\AA',Z')$. Up to shifting, we may assume $\AA'$ is a tilt of $\AA(\alpha,-\frac{1}{2})$. Since by Lemma \ref{lemma_restriction}, the heart $\L_{\OO_X}\Phi(\AA)$ is a tilt of $\AA(\alpha,-\frac{1}{2})$, by \cite[Lemma 8.11]{BMS} we conclude $\sigma'=(\L_{\OO_X} \circ \Phi) \cdot \sigma(\alpha,-\frac{1}{2})$, as we claimed.
\end{proof}

As a direct consequence, we obtain the following property of the Serre functor.
\begin{cor}
\label{cor_serrefunctorc3}
The Serre functor of $\Ku(X)$ preserves the orbit $\KK$.
\end{cor}
\begin{proof}
By \eqref{eq_Serrefunctor} it is enough to prove the statement for $\L_{\OO_X} \circ \Phi$. But this is a consequence of Proposition \ref{prop_twistfctrvsstabcond} and the fact that the action of $\tilde{\text{GL}}^+_2(\R)$ commutes with the action of $\Aut(\Ku(X))$.
\end{proof}

\begin{rmk}
We point out that the proof of Proposition \ref{prop_twistfctrvsstabcond}, and in particular of the previous lemmas, works for every Fano threefold of Picard rank $1$ and index $2$. Indeed, we have never used the fact that $X$ is a cubic threefold. Moreover, by \cite[Proposition 3.8]{Kuz04} and $\omega_X=\OO_X(-2H)$, the Serre functor of $\Ku(X)$ still satisfies the relation 
$$S_{\Ku(X)}^{-1}=(\L_{\OO_X} \circ \Phi) \circ (\L_{\OO_X} \circ \Phi)[-3].$$
In this paper, we only need the result for $d=3$, but we state the more general proposition for the reader convenience.
\begin{prop}
Let $X$ be a Fano threefold of Picard rank $1$ and index $2$. Then there exists $\tilde{g} \in \tilde{\emph{GL}}^+_2(\R)$ such that 
$$(\L_{\OO_X} \circ \Phi) \cdot \sigma(\alpha,-\frac{1}{2})= \sigma(\alpha,-\frac{1}{2}) \cdot \tilde{g}.$$
In particular, the Serre functor of $\Ku(X)$ preserves the orbit $\KK$. 
\end{prop}
\end{rmk}

\subsection{$S_{\Ku(X)}$-invariant stability conditions}
Let $X$ be a cubic threefold. Let us introduce the following notion.
\begin{dfn}
A stability condition $\sigma$ on $\Ku(X)$ is $S_{\Ku(X)}$-invariant if $S_{\Ku(X)} \cdot \sigma=\sigma \cdot \tilde{g}$ for $\tilde{g} \in \tilde{\text{GL}}^+_2(\R)$.
\end{dfn}

In the next lemmas, we prove some properties of the heart of a $S_{\Ku(X)}$-invariant stability condition.

\begin{lemma}
\label{lemma_boundphaseS}
For every $S_{\Ku(X)}$-invariant stability condition $\sigma$, if $F$ is a $\sigma$-semistable object of phase $\phi(F)$, then the phase of $S_{\mathsf{Ku}(X)}(F)$ satisfies $\phi(F) < \phi(S_{\mathsf{Ku}(X)}(F)) < \phi(F)+2$.
\end{lemma}
\begin{proof}
By definition, there exists $\tilde{g}=(g,M) \in \tilde{\text{GL}}^+_2(\R)$ such that $S_{\Ku(X)} \cdot \sigma = \sigma \cdot \tilde{g}$. Thus for every $\sigma$-semistable object $F$, the image $S_{\Ku(X)}(F)$ is $\sigma$-semistable with phase $\phi(S_{\Ku(X)} (F))=g(\phi(F))$. 

Assume $\phi(S_{\Ku(X)}(F)) \geq \phi(F)+2$. Then we have
$$\phi(S_{\Ku(X)}^2(F))=g(g(\phi(F))) \geq g(\phi(F)+2) =g(\phi(F))+2 \geq \phi(F)+4,$$
as $g$ is an increasing function. Similarly, we get $\phi(S_{\Ku(X)}^3(F)) \geq \phi(F)+6$, in contradiction with the fact that $\phi(S_{\Ku(X)}^3(F)) = \phi(F)+5$. Also $F$ and $S_{\Ku(X)}(F)$ cannot have the same phase with respect to $\sigma$, since otherwise we would have $\phi(S^3(F))=\phi(F)+5=\phi(F)$, which is impossible.
\end{proof}

\begin{lemma}
\label{lemma_homdimheartc3}
For every $S_{\Ku(X)}$-invariant stability condition $\sigma=(\AA,Z)$, the heart $\AA$ has homological dimension $2$.
\end{lemma}
\begin{proof}
Let $A, B \in \AA$. The vanishing $\Hom^{i}(A,B)=0$ for $i < 0$ follows from the property of heart of $\AA$. Note that by Lemma \ref{lemma_boundphaseS}, the $\sigma$-semistable factors of $S_{\Ku(X)}(A)$ have phase in the interval $(0,3)$. On the other hand, the $\sigma$-semistable factors of $B[i]$ have phase in the interval $(i, i+1]$. By Serre duality, we deduce that
$$\Hom(A,B[i])= \Hom(B[i],S_{\Ku(X)}(A))=0 \quad \text{for } i \geq 3,$$ 
since the semistable factors of $B[i]$ have phase greater than the phase of the semistable factors of $S_{\Ku(X)}(A)$. This proves our claim.
\end{proof}

\begin{lemma}
\label{lemma_nohom1=0,1c3}
For every $S_{\Ku(X)}$-invariant stability condition $\sigma=(\AA,Z)$, there are no nonzero objects $A \in \AA$ with $\Hom^1(A,A)=0$ or $\Hom^1(A,A)\cong \C$.
\end{lemma}
\begin{proof}
This is the analogous of \cite[Lemma 4.4]{BMMS}. By \cite[Proposition 2.7(ii)]{BMMS} we know that $\chi(A,A) \leq -1$. Then if $A \in \AA$, by Lemma \ref{lemma_homdimheartc3}, we must have $\hom^1(A,A) \geq 2$.
\end{proof}

The previous lemmas allow to prove a weak version for cubic threefolds of the Mukai Lemma proved in \cite[Lemma 2.5]{BaBri} for K3 surfaces. Set $\hom^1(A,A):=\text{dim}\Hom^1(A,A)$.
\begin{lemma}[Weak Mukai Lemma]
\label{lemma_Mukailemmac3}
Let $\sigma$ be a $S_{\Ku(X)}$-invariant stability condition. Let $A \to E \to B$ be a triangle in $\mathsf{Ku}(X)$ such that $\Hom(A,B)=0$ and  the $\sigma$-semistable factors of $A$ have phases greater or equal than the phases of the $\sigma$-semistable factors of $B$. Then
$$\hom^1(A,A)+ \hom^1(B,B) \leq \hom^1(E,E).$$
\end{lemma}
\begin{proof}
By Serre duality, we have $\Hom(B,A[2])=\Hom(A[2],S_{\Ku(X)}(B))$. Assume for simplicity that $A$ and $B$ are $\sigma$-semistable. Since $\sigma$ is $S_{\Ku(X)}$-invariant and by Lemma \ref{lemma_boundphaseS}, the object $S_{\Ku(X)}(B)$ is $\sigma$-semistable with phase $\phi(S_{\Ku(X)}(B))< \phi(B)+2$. Thus $\phi(A[2]) \geq \phi(B)+2> \phi(S_{\Ku(X)}(B))$, which implies the vanishing $\Hom(A[2],S_{\Ku(X)}(B))=0$. If $A$ and $B$ are not $\sigma$-semistable, applying the same argument to their semistable factors, we get the required vanishing. Then the argument of \cite[Lemma 2.5]{BaBri} applies to this setting, implying the result.  
\end{proof}

The weak Mukai Lemma implies the stability of objects with $\hom^1(E,E)=2$.

\begin{lemma}
\label{lemma_stabobjectsext1=2}
Let $\sigma$ be a $S_{\Ku(X)}$-invariant stability condition. Then every $E \in \Ku(X)$ with $\hom^1(E,E)=2$ is $\sigma$-stable.
\end{lemma}
\begin{proof}
Assume $E$ is unstable with respect to $\sigma$. Then there is a triangle 
$A \to E \to B$
in $\Ku(X)$, where $B \in \PP(\phi)$ and $A \in \PP(>\phi)$. Since $\Hom(A,B)=0$, by the weak Mukai Lemma \ref{lemma_Mukailemmac3}, we have 
$$\hom^1(A,A)+ \hom^1(B,B) \leq \hom^1(E,E)=2.$$
Note that $B$ cannot have $\hom^1(B,B)=0$ or $1$ by Lemma \ref{lemma_nohom1=0,1c3}. Moreover, if $\hom^1(A,A)=0$, then all its $\sigma$-semistable factors would satisfy the same property by Lemma \ref{lemma_Mukailemmac3}, in contradiction with Lemma \ref{lemma_nohom1=0,1c3}. It follows that $E$ is $\sigma$-semistable. 

Now assume that $E$ is strictly $\sigma$-semistable of phase $\phi$. Up to shifting, we may assume that $E$ is in the heart of $\sigma$. Assume firstly that $E$ has at least two non-isomorphic stable factors. Then we have a sequence
$$0 \to A \to E \to B \to 0$$
in $\PP(\phi)$ with $\Hom(A,B)=0$. By the weak Mukai Lemma \ref{lemma_Mukailemmac3}, we have 
$$\hom^1(A,A)+ \hom^1(B,B) \leq \hom^1(E,E)=2.$$
Since $A$ and $B$ with this property cannot exist in the heart of $\sigma$ by Lemma \ref{lemma_nohom1=0,1c3}, we deduce that $E$ is $\sigma$-stable.

Consider now the case that $E$ has a unique stable factor $A$ up to isomorphism. Note that $\hom^1(A,A)=2$ by Lemma \ref{lemma_homdimheartc3}. Since $\hom^2(A,A)=\hom(A[2],S_{\Ku(X)}(A))=0$ by Lemma \ref{lemma_boundphaseS}, it follows that $\chi(A,A)=-1$. Then $-1 \leq \chi(E,E)=n^2\chi(A,A)=-n^2$ for a positive integer $n$, which is impossible unless $n=1$. This ends the proof of the claim and implies the statement.
\end{proof}

\begin{rmk}
By Lemma \ref{lemma_stabobjectsext1=2} the objects $\II_\ell$ and $\JJ_\ell$ are $\sigma$-stable for every $S_{\Ku(X)}$-invariant stability condition $\sigma$. Thus by \eqref{eq_phases} we deduce that the image of the central charge $Z$ of $\sigma$ is not contained in a line and $Z \in Z(\alpha_0,-\frac{1}{2}) \cdot \mathrm{GL}^+_2(\R)$.
\end{rmk}

\begin{rmk}
As suggested by the referee, we point out that all the arguments in this section work by replacing $\Ku(X)$ with any fractional Calabi-Yau category $\mathcal{D}$ of dimension $<2$ (see \cite{Kuz19} for the definition) with negative definite numerical K-theory.
\end{rmk}

\subsection{Applications}
By Corollary \ref{cor_serrefunctorc3} every $\sigma \in \KK$ is $S_{\Ku(X)}$-invariant. Thus the results of the previous section hold for $\sigma$ and allow to prove Theorem \ref{cor_smoothmod} and to give another proof of the categorical Torelli Theorem, firstly showed in \cite{BMMS}.

\begin{proof}[Proof of Theorem \ref{cor_smoothmod}]
Assume that $M_\sigma(\Ku(X),\kappa)$ is a non-empty moduli space of $\sigma$-stable objects in $\Ku(X)$ with numerical class $\kappa$, with $\sigma=(\AA,Z) \in \KK$. Consider $E \in M_\sigma(\Ku(X),\kappa)$. Up to shift, we may assume $E \in \AA$. By Lemma \ref{lemma_homdimheartc3}, we have the vanishing $\Hom^i(E,E)=0$ for every $i \neq 0,1,2$. By Serre duality, we have
$$\Hom^2(E,E) = \Hom(E[2],S_{\Ku(X)}(E))=0,$$
since by Corollary \ref{cor_serrefunctorc3} and Lemma \ref{lemma_boundphaseS} the object $S_{\Ku(X)}(E)$ is $\sigma$-stable with phase $< \phi(E)+2$. Since $E$ is stable, we have that $\hom^1(E,E)=1-\chi(E,E)$ is constant. This proves that $M_\sigma(\Ku(X),\kappa)$ is smooth, as we wanted.
\end{proof}

For the categorical Torelli Theorem, we need this stronger version of Theorem \ref{Fanolines_modspace}.

\begin{lemma}
\label{lemma_Fanovariety}
Let $\sigma$ be a $S_{\Ku(X)}$-invariant stability condition on $\Ku(X)$. Then $M_\sigma(\Ku(X),[\II_\ell]) \cong \Sigma(X)$.
\end{lemma}
\begin{proof}
Let $E$ be a $\sigma$-stable object with $[E]=[\II_\ell]$. Then $\chi(E,E)=-1$. The same argument in the proof of Theorem \ref{cor_smoothmod} implies that $\hom^1(E,E)=2$. Thus by Lemma \ref{lemma_stabobjectsext1=2} $E$ is $\sigma(\alpha,\beta)$-stable. By Proposition \ref{lem_numclassI} we deduce that $E \cong \II_\ell$ for some line $\ell \subset X$ up to shifting. Together with Lemma \ref{lemma_stabobjectsext1=2} this implies a bijection between the Fano surface of lines $\Sigma(X)$ and $M_\sigma(\Ku(X),[\II_\ell])$. By \cite[Section 5.2]{BMMS} this bijection defines an isomorphism of algebraic varieties.
\end{proof}

\begin{thm}[\cite{BMMS}, Theorem 1.1]
\label{thm_catTorelli}
Two cubic threefolds $X$ and $X'$ are isomorphic if and only if there is an exact equivalence between $\Ku(X)$ and $\Ku(X')$.
\end{thm}
\begin{proof}
Assume there is an exact equivalence $\Phi: \Ku(X) \xrightarrow{\sim} \Ku(X')$. By \cite[Lemma 2.8]{BMMS}, up to composing with a power of the Serre functor of $\Ku(X)$, we may assume $[\Phi_*(\II_{\ell})]=[\II_{\ell'}]$ for $\ell$, $\ell'$ lines in $X$ and $X'$, respectively. Set $\sigma:=\sigma(\alpha,\beta) \in \Stab(\Ku(X))$; by Theorem \ref{Fanolines_modspace} and our assumption we have
$$\Sigma(X) \cong M_\sigma(\Ku(X),[\II_{\ell}]) \cong M_{\Phi \cdot \sigma}(\Ku(X'),[\II_{\ell'}]).$$
Note that $\Phi \cdot \sigma$ is a $S_{\Ku(X')}$-invariant stability condition. Indeed, we have $S_{\Ku(X')} \cdot (\Phi \cdot \sigma)=\Phi \cdot (S_{\Ku(X)} \cdot \sigma)$ since the Serre functors commute with equivalences (see \cite[Lemma 1.30]{Huy}). Then we have $\Phi \cdot (S_{\Ku(X)} \cdot \sigma)= \Phi \cdot (\sigma \cdot \tilde{g})=(\Phi \cdot \sigma) \cdot \tilde{g}$ by Corollary \ref{cor_serrefunctorc3} and the fact that the actions of $\Phi$ and $\tilde{\text{GL}}^+_2(\R)$ commute. Then by Lemma \ref{lemma_Fanovariety} we deduce that $M_{\Phi \cdot \sigma}(\Ku(X'),[\II_{\ell'}]) \cong \Sigma(X')$. Since the canonical bundle of the Fano surface of a cubic threefold is identified with the Pl\"ucker polarization, by \cite[Proposition 4]{Charles} we conclude that $X \cong X'$.
\end{proof}

\section{Quartic double solids}

Let $X$ be the double cover of $\P^3$ ramified in a quartic surface. By \cite[Corollary 4.6]{Kuz19}, the Serre functor of $\Ku(X)$ is
$$S_{\Ku(X)}= \iota [2],$$
where $\iota$ is the autoequivalence of $\Ku(X)$ induced by the involution of the double covering. We firstly study the action of $\iota$ on $\KK$ and on its boundary.

\begin{lemma}
\label{lemma_iotapreseK}
The involution $\iota$ acts as the identity on the closure $\overline{\KK}$ of $\KK$ in $\Stab(\Ku(X))$.
\end{lemma}
\begin{proof}
Note that $\iota \cdot \sigma(\alpha,\beta)=\sigma(\alpha,\beta)$. Indeed, $\iota$ preserves $\Coh(X)$ and the Chern character. Since the action of $\tilde{\text{GL}}_2^+(\R)$ commutes with autoequivalences, the previous observation implies that $\iota$ acts as the identity on $\KK$. Since the action of $\iota$ on the stability manifold is continuous, we deduce the statement.
\end{proof}
 
Lemma \ref{lemma_iotapreseK} allows to prove the following properties, in analogy to what was done in Section 5.2.  
\begin{lemma}
\label{lemma_homdimheart}
For every $\sigma=(\AA,Z)$ in $\overline{\KK}$, the heart $\AA$ has homological dimension $2$.
\end{lemma}
\begin{proof}
Let $A, B$ be in $\AA$. The vanishing of $\Hom^{i}(A,B)=0$ for $i <0$ follows from the property of heart of $\AA$. By Lemma \ref{lemma_iotapreseK}, $\iota$ preserves the phases of stable factors and the $\sigma$-semistable factors of $\iota(A)$ are the image via $\iota$ of the semistable factors of $A$. Thus for every $i > 2$ we have
$$\Hom^i(A,B)=\Hom(B,\iota(A)[2-i])=0.$$
This proves that $\AA$ has homological dimension $2$.
\end{proof}

\begin{lemma}
\label{lemma_extdimheart}
For every $\sigma=(\AA,Z)$ in $\overline{\KK}$, there are no nonzero objects $A \in \AA$ with $\hom^1(A,A)=0$ or $1$.
\end{lemma}
\begin{proof}
Using the computation of the pairing on the numerical K-theory of $\Ku(X)$ in \cite{Kuz}, we get $\chi(A,A) \leq -1$. Then if $A \in \AA$, by Lemma \ref{lemma_homdimheart}, we must have $\hom^1(A,A) \geq 2$. 
\end{proof}

In order to prove Theorem \ref{thm_conncomp}, we finally need a version of the Mukai Lemma in this setting.
\begin{lemma}[Weak Mukai Lemma]
\label{Mukailemma}
Let $A \to E \to B$ be a triangle in $\mathsf{Ku}(X)$ with $\Hom(A,B)=\Hom(A,\iota(B))=0$. Then
$$\hom^1(A,A)+ \hom^1(B,B) \leq \hom^1(E,E).$$
\end{lemma}
\begin{proof}
By Serre duality $\Hom(B,A[2])=\Hom(A,\iota(B))=0$. Then the same argument of \cite[Lemma 2.5]{BaBri} applies, implying the result.
\end{proof}

\begin{proof}[Proof of Theorem \ref{thm_conncomp}]
Assume that $\DD$ is a connected component of $\Stab(\Ku(X))$ strictly containing $\KK$. Then $\DD$ contains the boundary of $\KK$ in $\Stab(\Ku(X))$, which consists of stability conditions whose central charge is not injective, by Remark \ref{rmk_almostcc}. Denote by $\sigma=(\PP,Z)$ one of them. By the support property, we may assume that there is $\sigma_1:=\tilde{g} \cdot \sigma(\alpha_0,-\frac{1}{2})$ such that if $E$ is $\sigma_1$-stable of phase $\phi$, then $E \in \PP(\phi-\varepsilon, \phi+\varepsilon)$. By assumption, the image of $Z$ is contained in a line; this implies $E \in \PP(\theta)$ for a certain $\phi-\varepsilon < \theta < \phi+\varepsilon$. As a consequence, $\II_\ell$ and $\JJ_\ell$ are $\sigma$-semistable, by Proposition \ref{lemma_idealsheafstab} and Proposition \ref{prop_stabJ}.  

We claim that $\II_\ell$ and $\JJ_\ell$ are $\sigma$-stable. As a consequence, we would have that \eqref{eq_phases} holds for $\sigma$, contradicting the fact that $Z$ has image in a line. This implies that such degenerate $\sigma$ cannot exist on the boundary of $\KK$, proving that $\KK=\DD$.

It remains to prove the claim. Assume $E= \II_\ell$ or $\JJ_\ell$ is strictly $\sigma$-semistable of phase $\phi$. If all stable factors of $E$ are isomorphic, we get a contradiction with $\chi(E,E)=-1$. Assume that $E$ has at least two non-isomorphic stable factors. Then by a standard argument (see \cite[Proof of Lemma 2.6]{BaBri} or \cite[(2.4)]{HMS}) we can write a sequence in $\PP(\phi)$ of the form
$$0 \to A \to E \to B \to 0,$$
where $A$, $B$ are $\sigma$-semistable, all the stable factors of $B$ are isomorphic and $\Hom(A,B)=0$. If $\Hom(A, \iota(B))=0$, then Lemmas \ref{lemma_extdimheart} and \ref{Mukailemma} imply the stability of $E$ by the same argument as in Lemma \ref{lemma_stabobjectsext1=2}. Assume that $\Hom(A, \iota(B)) \neq 0$. By Lemma \ref{lemma_iotapreseK} the object $\iota(C)$ is $\sigma$-stable of the same phase of $C$. Then we have $\Hom(A,\iota(C)) \neq 0$. Moreover, $\iota(C)$ is a quotient of $A$ and $C \ncong \iota(C)$. We can write a sequence $D \to A \to F$, where $D$ and $F$ are $\sigma$-semistable, all the stable factors of $F$ are isomorphic to $\iota(C)$, and $C$, $\iota(C)$ are not stable factors of $D$. Then we have the following commutative diagram 
$$
\xymatrix{
D \ar[d]^{\id} \ar[r] & A \ar[d] \ar[r] & F \ar[d]\\
D \ar[d] \ar[r] & E \ar[d] \ar[r] & G \ar[d] \\
0 \ar[r] & B \ar[r]^{\id} & B,
}
$$
where $G$ is $\sigma$-semistable and whose stable factors are isomorphic to $C$ or $\iota(C)$. It follows that $\Hom(D,G)=\Hom(D,\iota(G))=0$ by Lemma \ref{lemma_iotapreseK}. By Mukai Lemma \ref{Mukailemma}, we deduce that
$$\hom^1(D,D)+\hom^1(G,G) \leq \hom^1(E,E)=2.$$
Since this is impossible by Lemma \ref{lemma_extdimheart}, we must have $\Hom(A, \iota(B))=0$. This ends the proof of the statement.
\end{proof}

Dipartimento di Matematica ``F.\ Enriques'', Universit\`a degli Studi di Milano, Via Cesare Saldini 50, 20133 Milano, Italy. \\
\indent E-mail address: \texttt{laura.pertusi@unimi.it}\\
\indent URL: \texttt{http://www.mat.unimi.it/users/pertusi} \\

Center for Applied Mathematics, Tianjin University, Weijin Road 92, Tianjin 300072, P.\ R.\ China. \\
\indent E-mail address: \texttt{syangmath@tju.edu.cn}\\
\indent URL: \texttt{http://cam.tju.edu.cn/en/faculty/index.php?id=44}
\end{document}